\def\version{April 14, 2014}

\documentclass[12pt]{article}
\usepackage[textwidth=500pt,textheight=650pt,centering]{geometry} 
\usepackage{amsmath,amssymb,amsthm}
\usepackage{cite}

\title
{Structural stability of a dynamical system \\ near a non-hyperbolic fixed point}

\author{
  Roland Bauerschmidt\thanks{Department of Mathematics,
    University of British Columbia,
    Vancouver, BC, Canada V6T 1Z2.
    Present address: School of Mathematics,
    Institute for Advanced Study,
    Princeton, NJ 08540 USA.
    E-mail: {\tt brt@math.ias.edu}.},\;
  David  Brydges\thanks{Department of Mathematics,
    University of British Columbia,
    Vancouver, BC, Canada V6T 1Z2.
    E-mail: {\tt db5d@math.ubc.ca}, {\tt slade@math.ubc.ca}.}\;
  and Gordon Slade$^\dagger$}

\date\version

\usepackage{amsthm}
\usepackage{enumerate}
\usepackage{color}

\numberwithin{equation}{section}
\theoremstyle{plain}
\newtheorem{theorem}{Theorem}[section]
\newtheorem{lemma}      [theorem] {Lemma}
\newtheorem{prop}       [theorem] {Proposition}
\newtheorem{cor}        [theorem] {Corollary}

\theoremstyle{definition}
\newtheorem{defn}       [theorem] {Definition}
\newtheorem{example}    [theorem] {Example}
\newtheorem{rk}         [theorem] {Remark}

\newtheorem{ass}        {Assumption}

\newcommand{\lbeq}[1]  {\label{e:#1}}
\newcommand{\refeq}[1] {\eqref{e:#1}}
\newcommand{\nnb}   {\nonumber \\}

\newcommand{\Vcal}   {\mathcal{V}}
\newcommand{\Wcal}   {\mathcal{W}}

\newcommand{\R}{\mathbb{R}}

\newcommand{\N}{\mathbb{N}}

\newcommand{\1}{1}

\newcommand{\half}{\textstyle{\frac 12}}
\newcommand{\ddp}[2]{\frac{\partial #1}{\partial #2}}

\newcommand{\proj}{\pi}

\newcommand{\gbar}{\bar{g}}
\newcommand{\zbar}{\bar{z}}
\newcommand{\mubar}{\bar{\mu}}
\newcommand{\betamax}{\|\beta\|_\infty}
\renewcommand{\b}{b}
\renewcommand{\u}{}
\newcommand{\domK}{a}
\newcommand{\domV}{h}

\newcommand{\hK}{(\domK-\domK_*)}
\newcommand{\hV}{\domV}
\newcommand{\jm}{j_\Omega}
\newcommand{\Mext}{M_\mathrm{ext}}

\renewcommand{\v}{{\sf r}}
\newcommand{\w}{{\sf w}}
\newcommand{\B}{{\sf B}}
\newcommand{\I}{{\sf N}}

\newcommand{\g}{{\mathring{g}}}
\newcommand{\x}{\mathring{x}}

\newcommand{\Kr}{\mathring{K}}

\newcommand{\alr}{\mathring{\alpha}}
\newcommand{\sigmar}{\mathring{\sigma}}
\newcommand{\taur}{\mathring{\tau}}
\newcommand{\etar}{\mathring{\eta}} 
\newcommand{\gammar}{\mathring{\gamma}} 
\newcommand{\lambdar}{\mathring{\lambda}}
\newcommand{\xir}{\mathring{\xi}}
\newcommand{\etat}{\tilde{\eta}} 
\newcommand{\gammat}{\tilde{\gamma}} 
\newcommand{\lambdat}{\tilde{\lambda}}
\newcommand{\xit}{\tilde{\xi}}

\newcommand{\Sb}{S^0}

\begin{document}

\maketitle

\begin{abstract}
  We prove structural stability under perturbations for a class of
  discrete-time dynamical systems near a non-hyperbolic fixed point.
  We reformulate the stability problem in terms of the well-posedness
  of an in\-fi\-nite-di\-men\-sio\-nal nonlinear ordinary differential
  equation in a Banach space of carefully weighted sequences.
  Using this, we prove existence and regularity of flows of the
  dynamical system which obey mixed initial and final boundary
  conditions.  The class of dynamical systems we study, and the
  boundary conditions we impose, arise in a renormalisation group
  analysis of the 4-dimensional weakly self-avoiding walk
  and the 4-dimensional $n$-component $|\varphi|^4$ spin model.
\end{abstract}

\section{Introduction and main result}
\label{sec:intro}

\subsection{Introduction}
\label{sec:approx-dynsys}

Let $\Vcal = \R^3$ with elements $V \in \Vcal$ written $V=(g,z,\mu)$
and considered as a column vector for matrix multiplication.
For each $j\in\N_0=\{0,1,2,\ldots\}$, we define the \emph{quadratic flow}
$\bar\varphi_j : \Vcal \to \Vcal$ by
\begin{equation}
\lbeq{Mmatrix}
  \bar\varphi_j(V) =
  \begin{pmatrix}
    1 & 0 &0 \\
    0 & 1 & 0 \\
    \eta_j & \gamma_j & \lambda_j
  \end{pmatrix} V
  -
  \begin{pmatrix}
    V^T q^g_j V\\
    V^T q^z_j V\\
    V^T q^\mu_j V
  \end{pmatrix}
  ,
\end{equation}
with the quadratic terms of the form
\begin{equation}
  q_j^g  = \begin{pmatrix}
      \beta_j & 0 & 0\\
      0 & 0 & 0 \\
      0 & 0 & 0
      \end{pmatrix},
  \quad
  q_j^z = \begin{pmatrix}
      \theta_j & \half \zeta_j & 0\\
      \half \zeta_j & 0 & 0 \\
      0 & 0 & 0
    \end{pmatrix},
\lbeq{qdef}
\end{equation}
and
\begin{equation}
  q_j^\mu = \begin{pmatrix}
      \upsilon_j^{gg} & \half \upsilon_j^{gz} & \half \upsilon_j^{g\mu}\\
      \half \upsilon_j^{gz} & \upsilon_j^{zz} & \half \upsilon_j^{z\mu} \\
      \half \upsilon_j^{g\mu} & \half \upsilon_j^{z\mu} & 0
    \end{pmatrix}.
\lbeq{Qdef}
\end{equation}
All entries in the above matrices are real numbers.
Our hypotheses on the parameters of $\bar\varphi$
are stated precisely in Assumptions (A1--A2) below.
In particular, we assume that
there exists $\lambda > 1$ such that $\lambda_j \ge \lambda$ for all
$j$.

The quadratic flow $\bar\varphi$ defines a time-dependent discrete-time 3-dim\-en\-sion\-al
dynamical system.
It is triangular, in the sense that
the equation for $g$ does not depend on
$z$ or $\mu$, the equation for $z$ depends only on $g$,
and the equation for $\mu$ depends on
$g$ and $z$. Moreover, the equation for $z$ is linear in $z$,
and the equation for $\mu$ is linear in $\mu$.
This makes the analysis of the quadratic flow elementary.

Our main result concerns structural stability of $\bar \varphi$ under a class of
infinite-dimensional perturbations.
Let $(\Wcal_j)_{j\in\N_0}$ be a sequence of Banach spaces and $X_j = \Wcal_j \oplus \Vcal$.
We write $x_j \in X_j$ as $x_j=(K_j,V_j)=(K_j,g_j,z_j,\mu_j)$.
Suppose that we are given maps
$\psi_j : X_j \to \Wcal_{j+1}$ and $\rho_j : X_j \to \Vcal$.
Then we define $\Phi_j : X_j \to X_{j+1}$ by
\begin{equation}
\label{e:F}
  \Phi_j(K_j,V_j)
  = \big(\psi_j(K_j,V_j), \bar \varphi_j(V_j) + \rho_j(K_j,V_j)\big)
  .
\end{equation}
This is an infinite-dimensional perturbation of the 3-dimensional
quadratic flow $\bar \varphi$, which breaks triangularity and which involves the spaces
$\Wcal_j$ in a nontrivial way.
We impose estimates on $\psi_j$  and
$\rho_j$ below, which make $\Phi$ a third-order perturbation of $\bar\varphi$.

We give hypotheses under which there exists a
sequence $(x_j)_{j\in\N_0}$ with
$x_j \in X_j$ which is a \emph{global flow} of $\Phi$,
in the sense that
\begin{equation} \lbeq{flow}
  x_{j+1} = \Phi_j(x_j) \quad \text{for all $j \in \N_0$},
\end{equation}
obeying the boundary conditions that $(K_0,g_0)$ is fixed,
$z_j \to 0$, and $\mu_j \to 0$.
Moreover, within an appropriate space of sequences, this global flow is unique.

\medskip
As we discuss in more detail in Section~\ref{sec:mainres} below,
this result provides an essential ingredient in a renormalisation group
analysis of the 4-dimensional continuous-time weakly self-avoiding walk
\cite{BBS-saw4-log,BS11,BBS-saw4}, where the boundary condition $\lim_{j \to \infty} \mu_j = 0$
is the appropriate boundary condition for the study of a \emph{critical}
trajectory.
Similarly, our main result applies also to the analysis of the critical
behaviour of the $4$-dimensional $n$-component $|\varphi|^4$ spin model \cite{BBS-phi4-log}.
These applications provide our immediate
motivation to study the dynamical system $\Phi$, but we expect
that the methods developed here will have further applications to dynamical
systems arising in
renormalisation group analyses in statistical mechanics.

\subsection{Dynamical system}

We think of $\Phi = (\Phi_j)_{j\in\N_0}$ as the \emph{evolution map} of a discrete
time-dependent dynamical system, although it is more usual in
dynamical systems to have the spaces $X_j$ be identical.
Our application in \cite{BBS-phi4-log,BBS-saw4-log,BS11,BBS-saw4} requires the greater generality
of $j$-dependent spaces.

In the case that $\Phi$ is a time-independent dynamical system,
i.e., when $\Phi_j = \Phi$ and $X_j=X$ for all $j\in \N_0$,
its fixed points are of special interest: $x^* \in X$ is a fixed point of $\Phi$ if $x^*=\Phi(x^*)$.
The dynamical system is called \emph{hyperbolic} near a fixed point $x^* \in X$ if the spectrum of $D\Phi(x^*)$
is disjoint from the unit circle \cite{Ruel89}.
It is a classic result that for a hyperbolic system there exists a splitting $X\cong X_s \oplus X_u$ into a \emph{stable} and an
\emph{unstable manifold} near $x^*$. The stable manifold is a submanifold $X_s \subset X$
such that $x_j \to x^*$ in $X$, exponentially fast, when $(x_j)$ satisfies \eqref{e:flow} and $x_0 \in X_s$.
On the other hand, trajectories started on the unstable manifold move away
from the fixed point.
This result can  be generalised without much difficulty to the situation
when the $\Phi_j$ and $X_j$ are not necessarily identical, viewing ``$0$''
as a fixed point (although $0$ is the origin in different spaces $X_j$).
The hyperbolicity condition must now be imposed in a uniform way \cite[Theorem 2.16]{Bryd09}.

By definition, $\bar\varphi_j(0)=0$,
and we will make assumptions below which can be interpreted as a
weakened formulation
of the fixed point equation $\Phi_j(0)=0$
for the dynamical system defined by \refeq{F}.
Nevertheless, for simplicity we refer to $0$ as a fixed point of
$\Phi = (\Phi_j)$.
This fixed point $0$ is not hyperbolic
due to the two unit eigenvalues of the matrix in the first term of \refeq{Mmatrix}.
Thus the $g$- and $z$-directions are \emph{centre} directions,
which neither contract nor expand in a linear approximation.
On the other hand, the hypothesis that $\lambda_j \ge \lambda >1$ ensures
that the $\mu$-direction is \emph{expanding}, and we will assume below
that $\psi_j : X_j \to \Wcal_{j+1}$ is such that the
$K$-direction is \emph{contractive} near the fixed point $0$.
The behaviour of dynamical systems near non-hyperbolic fixed points
is much more subtle than for the hyperbolic case.
A general classification does not exist, and a nonlinear analysis is
required.

\subsection{Main result}
\label{sec:mainres}

In Section~\ref{sec:approximate-flow}, we give an elementary proof that
for $\gbar_0$ positive and sufficiently small,
there exists a unique choice of $(\bar{z}_0,\bar{\mu}_0)$ such
that the global flow $\bar V = (\gbar,\zbar,\mubar)$ of
$\bar\varphi$ satisfies
$(\zbar_\infty,\mubar_\infty) = (0,0)$,
where we write, e.g., $\zbar_\infty =\lim_{j\to\infty}\zbar_j$.
Our main result is that,
under the assumptions stated below,
there exists a unique global flow of $\Phi$ with any small initial condition
$(K_0,g_0)$ and with final condition $(\zbar_\infty,\mubar_\infty) = (0,0)$,
and that this flow is
a small perturbation of $\bar V$.

The sequence $\bar g = (\gbar_j)$ plays a prominent role in the analysis.
Determined by the sequence $(\beta_j)$, it obeys
\begin{equation}
\lbeq{gbar}
    \gbar_{j+1} = \gbar_j - \beta_j \gbar_j^2,
    \qquad \bar g_0 = g_0>0.
\end{equation}
We regard $\gbar$ as a known sequence (only dependent on
the initial condition $g_0$).
The following examples are helpful to keep in mind.

\begin{example}
\label{ex:chi}
(i)
Constant $\beta_j=b >0$.
In this case,
$\gbar_j \sim g_0(1+g_0b j)^{-1} \sim (b j)^{-1}$ as $j \to \infty$
(an argument for this standard fact is outlined in the proof of
\refeq{chigbd2} below).

\smallskip \noindent (ii)
Abrupt cut-off, with $\beta_j=b$
for $j \le J$ and $\beta_j=0$ for $j> J$, with $J \gg 1$.
In this case, $\gbar_j$
is approximately the constant $(b J)^{-1}$ for $j>J$.
In particular, $\gbar_j$ does not go to zero as $j \to \infty$.
\end{example}

Example~\ref{ex:chi} prompts us to make the following general definition
of a cut-off time for bounded sequences $\beta_j$.
Let $\betamax = \sup_{j \ge 0} |\beta_j|< \infty$, and let
$n_+ = n$ if $n \ge 0$ and otherwise $n_+=0$.
Given a fixed $\Omega > 1$, we define the $\Omega$-\emph{cut-off time} $\jm$ by
\begin{equation}
    \jm =
    \inf \{ k \geq 0 : |\beta_j| \le \Omega^{-(j-k)_+} \betamax \;
    \text{for all } j \ge 0\}.
\end{equation}
The infimum of the empty set is defined to equal $\infty$, e.g., if
$\beta_j=b$ for all $j$ then $\jm=\infty$.
By definition, $\jm \le j_{\Omega'}$ if $\Omega \le \Omega'$.
To abbreviate the notation, we write
\begin{equation}
\lbeq{chidef}
    \chi_{j} = \Omega^{-(j-\jm)_+}.
\end{equation}

The evolution maps $\Phi_j$ are specified by the real parameters
$\eta_j$, $\gamma_j$, $\lambda_j$, $\beta_j$, $\theta_j$, $\zeta_j$,
$\upsilon_j^{\smash{\alpha \beta}}$, together
with the maps $\psi_j$ and $\rho_j$ on $X_j$.
Throughout this paper, we
fix $\Omega >1$ and make Assumptions~(A1--A2) on the real parameters
and Assumption~(A3) on the maps, all stated below.
The constants in all estimates are permitted to depend on the constants in these assumptions,
including $\Omega$, but \emph{not} on $\jm$ and $g_0>0$.
Furthermore, we consider the situation when the parameters of $\bar\varphi_j$
are continuous maps from a metric space $\Mext$ of external parameters, $m \in \Mext$, into $\R$,
that the maps $\psi_j$ and $\rho_j$ similarly have continuous dependence on $m$,
and that $\jm$ is allowed to depend on $m$.
In this situation,
the constants in Assumptions~(A1--A3) are assumed to hold uniformly in $m$.

\medskip

\begin{ass}
  \emph{The sequence $\beta$:}
  The sequence $(\beta_j)$ is bounded:  $\betamax < \infty$.
  There exists $c>0$ such that $\beta_j \ge c$ for all but $c^{-1}$ values of $j \le \jm$.
\end{ass}

\begin{ass}
  \emph{The other parameters of $\bar\varphi$:}
  There exists $\lambda>1$ such that $\lambda_j \ge \lambda$ for all $j \geq 0$.
  There exists $c>0$ such that
  $\zeta_j \leq 0$ for all but $c^{-1}$ values of $j \leq \jm$.
  Each of
  $\eta_j$, $\gamma_j$, $\theta_j$, $\zeta_j$, $\upsilon_j^{\smash{\alpha \beta}}$
  is bounded in absolute value by $O(\chi_j)$,
  with a constant that is independent of both $j$ and $\jm$.
\end{ass}

Note that when $\jm<\infty$, Assumption~(A1) permits the possibility
that eventually $\beta_k=0$ for large $k$.
The simplest setting for the assumptions is for the case $\jm=\infty$,
for which
$\chi_j=1$ for all $j$.
Our applications include situations in which $\beta_j$ approaches a positive
limit as $j\to\infty$, but also situations in which $\beta_j$ is
approximately constant in $j$ over a long initial interval $j \le \jm$ and then
abruptly decays to zero.

In Section~\ref{sec:approximate-flow}, in preparation for the proof of the main
result, we prove the following elementary proposition
concerning flows of the 3-dimensional quadratic dynamical system $\bar\varphi$.

\begin{prop} \label{prop:approximate-flow}
  Assume (A1--A2).
  If $\gbar_0>0$ is sufficiently small,
  then there exists a unique global flow $\bar V = (\bar V)_{j\in \N_0} =
  (\gbar_j, \zbar_j, \mubar_j)_{j\in\N_0}$
  of $\bar\varphi$ with initial condition
  $\gbar_0$ and
  $(\zbar_\infty , \mubar_\infty) = (0, 0)$.
  This flow satisfies $\gbar_j=O(\gbar_0)$ and
  \begin{equation} \lbeq{approximate-flow}
    \chi_j \gbar_j^n
    = O\left(\frac{\gbar_0}{1+\gbar_0j}\right)^n,
    \quad
    \zbar_j = O(\chi_j \gbar_j) ,
    \quad
    \mubar_j = O(\chi_j \gbar_j),
  \end{equation}
  with constants independent of $j_\Omega$ and $\gbar_0$, and with the first
  estimate valid for real $n \in [1,\infty)$ with an $n$-dependent constant.
  Furthermore, $\bar V_j$ is continuously
  differentiable in the initial condition $\gbar_0$,
  for every $j \in \N_0$, and
  if the maps $\bar\varphi_j$ depend continuously on an external parameter
  such that (A1--A2) hold with uniform constants, then $\bar V_j$ is continuous in this parameter,
  for every $j \in \N_0$.
\end{prop}

In particular, by \refeq{approximate-flow}, above scale $\jm$ each of $\zbar_j,\mubar_j$ decays exponentially.
We now define domains $D_j \subset X_j$ on which the perturbation $(\psi_j,\rho_j)$ is assumed to be defined,
and an assumption which states estimates for $(\psi_j,\rho_j)$. The domain and estimates depend on an initial
condition $g_0$ and the possible external parameter $m$.
For parameters $\domK,\domV>0$ and sufficiently small $g_0>0$, let $(\gbar_j,\zbar_j,\mubar_j)_{j\in\N_0}$
be the sequence determined by Proposition~\ref{prop:approximate-flow} with initial condition $\gbar_0=g_0$,
and define the domain $D_j = D_j(g_0, \domK,\domV) \subset X_j$  by
\begin{align}
  \lbeq{Djdef}
  D_j =
  \{
  x_j \in X_j :
  \|K_{j}\|_{\Wcal_j} &\leq \domK \chi_j\gbar_j^3,\nnb
  |g_j - \gbar_j| &\leq \domV \gbar_j^2 |\log \gbar_j|,\nnb
  |z_j - \zbar_j| &\leq \domV \chi_j\gbar_j^2 |\log \gbar_j|,\nnb
  |\mu_j - \mubar_j| &\leq \domV \chi_j\gbar_j^2 |\log \gbar_j|
  \}.
\end{align}
Note that if $\beta_j$ depends on an external parameter $m$, then the domain $D_j = D_j(g_0,m,\domK,\domV)$
also depends on this parameter $m$ through $\gbar_j =
\gbar_j(m)$.

Throughout the paper, we write
$D_\alpha\phi$ for the Fr\'echet derivative of a map $\phi$ with respect to the
component $\alpha$, and $L^m(X_j, X_{j+1})$ for the space of bounded $m$-linear maps
from $X_j$ to $X_{j+1}$.
The following assumption depends on positive parameters
$(g_0,\domK,\domV,\kappa,\Omega,R,M)$.
The norm $\|\cdot\|_{\Vcal}$ is the supremum norm on $\R^3$.
\begin{ass}
  \emph{The perturbation:}
  The maps
  $\psi_j : D_j \to \Wcal_{j+1} \subset X_{j+1}$ and
  $\rho_j : D_j \to \Vcal \subset X_{j+1}$
  are three times continuously Fr\'echet differentiable,
  there exist $\kappa \in (0,\Omega^{-1})$,
  $R \in (0,\domK(1-\kappa\Omega))$,
  and a constant $M>0$
  such that,
  for all $x_j = (K_j,V_j) \in D_j$,
    \begin{gather}
      \lbeq{R-bound}
      \|\psi_j(0, V_j)\|_{\Wcal_{j+1}} \le R\chi_{j+1}\gbar_{j+1}^3, \quad
      \|\rho_j(x_j)\|_{\Vcal} \le M \chi_{j+1}\gbar_{j+1}^3,
      \\
      \label{e:R-Lip-bound-K}
      \|D_K \psi_{j}(x_j)\|_{L(\Wcal_j,\Wcal_{j+1})} \leq \kappa,
      \quad
      \|D_K \rho_{j}(x_j)\|_{L(\Wcal_j,\Vcal)} \leq M,
      \\
      \intertext{and such that, for both $\phi=\psi$ and $\phi =\rho$
        and $2 \leq n+m \leq 3$,
      }
      \label{e:R-Lip-bound-V}
      \|D_V \phi_j(x_j)\|_{L(\Vcal, X_{j+1})}
      \leq M \chi_{j+1} \gbar_{j+1}^2, 
      \\
      \|D_V^mD_K^n  \phi_j(x_j)\|_{L^{n+m}(X_j, X_{j+1})}
      \leq M (\chi_{j+1} \gbar_{j+1}^3)^{1-n}  (\gbar_{j+1}^2 |\log \gbar_{j+1}|)^{-m}
      .
      \label{e:R-Lip-bound-higher}
    \end{gather}
  \end{ass}

The bounds \refeq{R-bound} guarantee that $\Phi$ is a third-order perturbation of $\bar\varphi$.
Moreover, since $\kappa <1$, the $\psi$-part of \refeq{R-Lip-bound-K}
ensures that the $K$-direction  is contractive for $\Phi$.
The bounds \refeq{R-Lip-bound-higher} permit the second and third derivatives
of $\psi$ and $\rho$ to be quite large.
The restriction on $R$ in (A3) may seem unnatural initially, but its role
is seen in Lemma~\ref{lem:small-ball} below.

The following elementary lemma provides a statement of
domain compatibility which shows that a sequence $(\bar K_j)_{j\in\N_0}$
can be defined inductively by $\bar K_{j+1} = \psi_j(\bar K_j, \bar V_j)$.
Denote by $\pi_K D_j$ the projection of $D_j$ onto $\Wcal_j$, i.e.,
\begin{equation}
  \pi_KD_j
  =
  \{
  K_j \in \Wcal_j : \|K_{j}\|_{\Wcal_j} \leq r \chi_j\gbar_j^3
  \} .
\end{equation}

\begin{lemma} \label{lem:small-ball}
  Assume (A3), let $\domK^* \in (R/(1-\kappa\Omega),\domK]$, and assume that $g_0>0$ is sufficiently small. Then
  $\psi_j(D_j(g_0, \domK^*,\domV)) \subseteq \pi_KD_{j+1}(g_0,\domK^*,\domV)$.
\end{lemma}

\begin{proof}
  The triangle inequality and the first bounds of \refeq{R-bound}--\refeq{R-Lip-bound-K} imply
  \begin{align}
    \|\psi_j(K_j,V_j)\|_{\Wcal_{j+1}}
    &\leq \|\psi_j(0,V_j)\|_{\Wcal_{j+1}} + \|\psi_j(K_j,V_j)-\psi_j(0,V_j)\|_{\Wcal_{j+1}}
    \nnb &
    \le R\chi_{j+1} \gbar_{j+1}^3 + \kappa \domK^* \chi_j \gbar_j^3.
  \end{align}
  Therefore,
  \begin{align}
    \|\psi_j(K_j,V_j)\|_{\Wcal_{j+1}}
    &\leq R\chi_{j+1} \gbar_{j+1}^3 + \domK^*\kappa\Omega(1+O(g_0))\chi_{j+1}\gbar_{j+1}^3
    \nnb
    &\leq \domK^* \chi_{j+1}\gbar_{j+1}^3
    ,
  \end{align}
  where the first inequality uses
  the facts that $\gbar_j^3/\gbar_{j+1}^3 = 1+O(g_0)$
  (verified in Lemma~\ref{lem:elementary-recursion}(i) below)
  and that $g_0>0$ is sufficiently small,
  and the second inequality uses the restriction on $R$ in (A3).
\end{proof}
The sequence $\bar x = (\bar K_j, \bar V_j)_{j\in\N_0}$ is
a flow of the dynamical system $\bar\Phi = (\psi,\bar\varphi)$ in the sense of \refeq{flow},
with initial condition $(\bar K_0, \gbar_0) = (K_0, g_0)$
and final condition $(\zbar_\infty, \mubar_\infty)=(0,0)$.
We consider this sequence as a function
$(K_0,g_0) \mapsto  \bar x(K_0,g_0)$
of the initial condition $(K_0,g_0)$.
Our main result is the following
theorem, which shows that flows $x$
of the dynamical system $\Phi = (\psi, \bar\varphi + \rho) = \bar\Phi + (0,\rho)$
are perturbations of the flows $\bar x$ of $\bar\Phi$.

\begin{theorem}
\label{thm:flow}
  Assume  (A1--A3) with parameters $(\domK,\domV,\kappa,\Omega,R,M)$
  and $g_0 = g_0'$,
  and let $\domK_* \in (R/(1-\kappa\Omega),\domK)$, $\b \in (0,1)$.
  There exists $\domV_* > 0$ such that for all $\domV \geq \domV_*$,
  there exists $g_*>0$ such that if $g_0' \in (0,g_*]$ and
  $\|K_0'\|_{\Wcal_0} \leq \domK_*g_0^3$, the following conclusions hold.
  \begin{enumerate}[(i)]
  \item
  There exists a neighbourhood
  $\I = \I(K_0',g_0') \subset \Wcal_0 \oplus \R$ of $(K_0',g_0')$
  such that, for initial conditions $(K_0, g_0) \in \I$,
  there exists a global flow $x$ of $\Phi=(\psi,\bar\varphi+\rho)$ with $(z_\infty,\mu_\infty)=(0,0)$
  such that, with $\bar x$ the unique flow of $\bar\Phi=(\psi,\bar\varphi)$ determined by the same
  boundary conditions,
  \begin{align}
     \lbeq{VVbar1}
     \| K_j -\bar K_j\|_{\Wcal_j}
     &\leq \b (\domK-\domK_*) \chi_j \gbar_j^3,
     \\
     \lbeq{VVbar2}
     |g_j - \gbar_j|
     &\leq \b \domV \gbar_j^2 |\log \gbar_j|,
     \\
     \lbeq{VVbar3}
     |z_j - \zbar_j|
     & \leq \b \domV \chi_j \gbar_j^2 |\log \gbar_j|     ,
     \\
     \lbeq{VVbar4}
     |\mu_j -\mubar_j|
     &\leq \b \domV \chi_j \gbar_j^2 |\log \gbar_j|
     .
  \end{align}
  The sequence $x$
  is the unique solution to
  \eqref{e:flow} which obeys
  these boundary conditions
  and the bounds \refeq{VVbar1}--\refeq{VVbar4}.

  \item
  For every $j \in \N_0$, the map
  $(K_j,V_j): \I \to \Wcal_j \oplus \Vcal$ is $C^1$
  and obeys
  \begin{equation} \lbeq{derivbd}
    \ddp{z_0}{g_0} = O(1), \quad \ddp{\mu_0}{g_0} = O(1).
  \end{equation}
\end{enumerate}
\end{theorem}

\begin{rk} \label{rk:Nrad}
  The proof of Theorem~\ref{thm:flow}
  shows that $\I$ contains a ball centred at $(K'_0,g'_0)$ whose
  radius depends only on $g'_0$, the constants in (A1--A3), and $a_*,b,h$,
  and that this radius is bounded below away from
  zero uniformly in $g_0'$ in a compact subset of $(0,g_*]$.
\end{rk}

Because of its triangularity, an exact analysis of the flows of $\bar\varphi$
with the boundary conditions of interest
is straightforward: the three equations for $g,z,\mu$
can be solved successively
and we do this in
Section~\ref{sec:approximate-flow} below.
Triangularity does not hold for $\Phi$, and
we prove in Sections~\ref{sec:pf}--\ref{sec:pf-DPhiyr} below that the flows of $\Phi$ with the boundary conditions of interest
nevertheless remain close to the flows of $\bar\varphi$ with the same boundary conditions.

\medskip

\begin{paragraph}{Application}

A fundamental element in renormalisation group analysis concerns
the flow of local interactions obtained via iteration of a renormalisation
group map \cite{WK74}.
The dynamical system \refeq{F} arises
as part of renormalisation group studies of
the critical behaviour of two different but related models:
the $4$-dimensional $n$-component $|\varphi|^4$ spin model \cite{BBS-phi4-log}, and the
$4$-dimensional  continuous-time weakly self-avoiding
walk \cite{BBS-saw4-log,BBS-saw4} (see \cite{BS11} for a preliminary version).
The main results
of \cite{BBS-saw4-log,BBS-saw4} are that,
for the continuous-time weakly self-avoiding walk in dimension four,
the susceptibility diverges with a logarithmic correction
as the critical point is approached,
and the critical two-point function has $|x|^{-2}$ decay.
Related results are obtained for the $4$-dimensional $n$-component
$|\varphi|^4$ spin model in \cite{BBS-phi4-log}, complementing and in some cases
extending results of \cite{FMRS87,GK86,Hara87,HT87}.
Theorem~\ref{thm:flow}
is an essential ingredient in analysing the flows in
\cite{BBS-phi4-log,BBS-saw4-log,BBS-saw4}, and the uniformity of
\refeq{VVbar1}--\refeq{VVbar4} in the cut-off time (for a given $\Omega$) is needed.
In \cite{BBS-phi4-log,BBS-saw4-log,BBS-saw4}, the index $j$ represents an increasingly large length
scale,
the spaces $\Wcal_j$ have a subtle definition and are of infinite dimension,
and their $j$-dependence is an inevitable consequence of applying the
renormalisation group to a lattice model.
\end{paragraph}

\begin{rk} \label{rk:flow}
  (i) For $\jm=\infty$ and with
  \refeq{approximate-flow}, the bounds \refeq{VVbar1}--\refeq{VVbar4}
  imply $\|K_j\|_{\Wcal_j} = O(j^{-3})$ and $V_j - \bar{V}_{j} = O(j^{-2} \log j)$.
  However, the latter bounds do not reflect that $K_j,V_j\to 0$
  as $g_0 \to 0$, while the former do.
  Furthermore, \refeq{approximate-flow} implies $\chi_j\gbar_j \to 0$
  as $j\to \infty$ (also when $\jm < \infty$), and thus \eqref{e:VVbar1}
  and \eqref{e:VVbar3}--\eqref{e:VVbar4} imply
  $K_j \to 0$, $z_j \to 0$, $\mu_j \to 0$ as $j\to\infty$.
  More precisely, these estimates imply
  $z_j,\mu_j = O(\chi_j\gbar_j)$ so that $z_j$ and $\mu_j$
  decay exponentially after the $\Omega$-cut-off time
  $\jm$; we interpret this as indicating
  that the boundary condition $(z_\infty,\mu_\infty)=(0,0)$ is essentially achieved
  already at $\jm$.

  \smallskip\noindent
  (ii)
  We conjecture that the error bounds in \eqref{e:VVbar1}--\eqref{e:VVbar4}
  have optimal decay as $j\to\infty$.  Some indication of this can be found
  in Remark~\ref{rk:weights} below.
\end{rk}

Theorem~\ref{thm:flow} is an analogue of a \emph{stable manifold
theorem} for the non-hyperbolic dynamical
system defined by \eqref{e:F}.  It is inspired by
\cite[Theorem~2.16]{Bryd09} which however holds only in
the hyperbolic setting.
Irwin \cite{Irwi70} showed that the stable manifold theorem for
hyperbolic dynamical systems
is a consequence of the implicit function theorem in Banach spaces
(see also \cite{Ruel89,Shub87}).
Irwin's approach was inspired by Robbin \cite{Robb68}, who showed
that the local existence
theorem for ordinary differential equations is a consequence of the
implicit function theorem.
By contrast,
in our proof of Theorem~\ref{thm:flow}, we directly apply the local
existence theorem for ODEs, without explicit mention of the implicit
function theorem.
This turns out to be advantageous to deal with the lack of hyperbolicity.

Our choice of $\bar\varphi$ in \refeq{Mmatrix} has a specific triangular
form.  One reason for this is that \refeq{Mmatrix} accommodates what
is required in our application in \cite{BBS-saw4-log,BS11,BBS-saw4,BBS-phi4-log}.
A second reason is that additional nonzero terms in $\bar \varphi$ can lead to the
failure of Theorem~\ref{thm:flow}.
The condition that $\beta_j$ is mainly non-negative is important for
the sequence $\gbar_j$ of
\eqref{e:gbar} to remain bounded.
The following example shows that
our sign restriction on the $\zeta_j$ term in the flow of
$\zbar$ is also important,
since positive $\zeta_j$ can lead to violation of a
conclusion of Theorem~\ref{thm:flow}.

\begin{example}
\label{ex:zeta-ass}
  Suppose that $\zeta_j = \theta_j= \beta_j = 1$,
  that $\rho=0$, and that $\gbar_0>0$ is small.
  For this constant $\beta$ sequence, $\jm=\infty$
  (for any $\Omega >1$) and hence $\chi_j =1$ for all $j$.
  As in Example~\ref{ex:chi}, $\gbar_j \sim j^{-1}$.
  By \refeq{Mmatrix} and \eqref{e:gbar},
  \begin{equation}
    \bar z_{j+1}
     = \bar z_j (1- \bar g_j) -\bar g_j^2
     = \bar z_j \frac{\bar g_{j+1}}{\bar g_j} - \bar g_j^2.
  \end{equation}
  Let $\bar y_j = \bar z_j/\bar g_j$.
  Since $\bar g_j / \bar g_{j+1} = (1-\gbar_j)^{-1} \ge 1$, we obtain
  $\bar y_j \ge \bar y_{j+1} + \bar g_j$ and hence
  \begin{equation}
    \bar y_j
    \ge
    \bar y_{n+1}  + \sum_{l=j}^n    \bar g_l
    .
  \end{equation}
  Suppose that $\bar z_j = O(\bar g_j)$, as in \refeq{approximate-flow}.
  Then $\bar y_j = O(1)$ and hence by
  taking the limit $n\to \infty$ we obtain
  \begin{equation}
    \bar y_j
    \geq \limsup_{n\to\infty}
    \left( \bar y_{n+1}  + \sum_{l=j}^n  \bar g_l \right)
    \geq -C + \sum_{l=j}^\infty \bar g_l.
  \end{equation}
  However, since $\gbar_j \sim j^{-1}$, the last sum diverges.
  This contradiction implies that the
  conclusion $\bar z_j = O(\bar g_j)$ of \refeq{VVbar3} is impossible.
\end{example}

\subsection{Continuity in external parameter}

The uniqueness statement of Theorem~\ref{thm:flow} implies the following
corollary regarding
continuous dependence on an external parameter of the global flow for \refeq{flow}
given by Theorem~\ref{thm:flow}.
In the statement of the corollary, we assume that $D_j$ is actually the union
over $m \in \Mext$ of the domains on the right-hand side of \eqref{e:Djdef}.
Recall that the latter domains depend on $m$ through $\beta$ and $\gbar$.

\begin{cor}
  \label{cor:masscont}
   Assume that $\Phi_j: D_j \times \Mext \to X_{j+1}$ are continuous maps
   and that (A1--A3) hold for $\Phi_j(\cdot, m)$, for every $m \in \Mext$,
   and with parameters independent of $m$.
   Let $x(m,u_0) = (K(m,u_0),V(m,u_0))$
   be the global flow for external parameter $m$ and initial condition $u_0=(K_0,g_0)$ guaranteed by Theorem~\ref{thm:flow}.
   Then $x_j$ is continuous in $(m,u_0)$ for each $j \in \N_0$.
\end{cor}

\begin{proof}
  We fix $m \in \Mext$, $u_0 \in \I$ and show that $x_j$ is continuous at this fixed $(m,u_0)$.
  For any $m' \in \Mext$, $u_0' \in \I$, let $x(m',u_0')=(K(m',u_0'),V(m',u_0'))$ denote the unique global
  flow of Theorem~\ref{thm:flow}; it satisfies the estimates
    \begin{align}
    \|K_j(m') - \bar K_j(m')\|_{\Wcal_j} \lbeq{VVbar1mp}
    &\leq \b (\domK-\domK_*) \chi_j(m') \gbar_j(m')^3
    ,
    \\
    |g_j(m') - \bar g_j(m')| \lbeq{VVbar2mp}
    &\leq \b \domV \gbar_j(m')^2 |\log\gbar_j(m')|
    ,
    \\
    |\mu_j(m') - \bar \mu_j(m')| \lbeq{VVbar3mp}
    &\leq \b \domV \chi_j(m') \gbar_j(m')^2 |\log\gbar_j(m')|
    ,
    \\
    |z_j(m') - \bar z_j(m')| \lbeq{VVbar4mp}
    &\leq \b \domV \chi_j(m') \gbar_j(m')^2 |\log\gbar_j(m')|
    .
  \end{align}
  By Proposition~\ref{prop:approximate-flow},
  $\bar V_j(m',g_0')$ is continuous in $(m',g_0')$, and thus in particular
  $\bar V_0(m',g_0')$ is uniformly bounded for $(m',g_0')$ in a
  bounded neighbourhood $I$ of $(m,g_0)$.
  With \refeq{VVbar2mp}--\refeq{VVbar4mp}, we see that
  $V_0(m',u_0')$ is therefore also uniformly bounded in $I$.
  Thus, for every sequence $(m',u_0') \to (m,u_0)$, $V_0(m',u_0')$ has a limit point.
  It suffices to show that the limit point is unique.
  To show this uniqueness, we fix an arbitrary limit point $V_0^*$
  and a sequence $(m', u_{0}') \to (m,u_0)$ such that $V_0(m', u_{0}') \to V_0^*$.
  Since $K_{0}' \to K_0$, we also set $K_0^*=K_0$.

  Then define $x_j^* = (K_j^*, V_j^*)$ by inductive application of $\Phi_j(m, \cdot)$
  starting from $x_0^* = (K_0^*,V_0^*)$, as long as $x_j^* \in D_j$.
  Since $(K_{0}',V_0(m',u_{0}')) \to (K_0^*,V_0^*)$, it follows by
  induction and the assumed continuity of $\psi_j,\rho_j$
  that $x_j(m',u_{0}') \to x_j^*$.
  By an analogous induction, using continuity of 
  $\bar V_j$ and $\psi_j$, 
  it follows that  $\bar K_j(m', u_{0}') \to \bar K_j(m,u_0)$.
  Since $\chi_j(m') \to \chi_j(m)$, we can now take the limit of
  \refeq{VVbar1mp}--\refeq{VVbar4mp} along the sequence $(m',u_{0}') \to (m,u_0)$
  and obtain
  \begin{align}
    \|K_j^* - \bar K_j(m)\|_{\Wcal_j} \lbeq{KKbarm}
    &\leq \b (\domK-\domK_*) \chi_j(m) \gbar_j(m)^3
    ,
    \\
    |g_j^* - \bar g_j(m)|
    &\leq \b \domV \gbar_j(m)^2 |\log\gbar_j(m)|
    ,
    \\
    |\mu_j^* - \bar \mu_j(m)|
    &\leq \b \domV \chi_j(m) \gbar_j(m)^2 |\log\gbar_j(m)|
    ,
    \\
    |z_j^* - \bar z_j(m)|
    &\leq \b \domV \chi_j(m) \gbar_j(m)^2 |\log\gbar_j(m)|
    .
  \end{align}
  The uniqueness assertion of Theorem~\ref{thm:flow} implies that $x_j^* = x_j(m,u_0)$,
  and we see that the above inductions can in fact be carried out indefinitely.
  We also conclude that $V_0^* = V_0(m,u_0)$, so there is a unique limit point of
  $V_0(m',u_0')$ as $(m',u_0') \to (m,u_0)$. This shows that $V_0$ is continuous
    at $(m,u_0)$.
  The continuity of $x_j$ now follows inductively from the continuity of the $\Phi_j$.
\end{proof}

\section{Quadratic flow}
\label{sec:approximate-flow}

In this section, we
prove that, for the quadratic approximation $\bar\varphi$,
there exists a unique solution $\bar V = (\bar V_j)_{j\in\N_0} =(\gbar_j,\zbar_j,\mubar_j)_{j\in\N_0}$
to the flow equation
\begin{equation} \lbeq{barflow}
  \bar V_{j+1} = \bar \varphi_j(\bar V_j)
  \quad \text{with fixed small $\gbar_0 >0$ and with
  $(\bar z_\infty, \bar \mu_\infty) = (0,0)$}
  .
\end{equation}
Due to the triangular nature of $\bar\varphi$,
we can obtain detailed information about the sequence $\bar V$.
In particular, we prove Proposition~\ref{prop:approximate-flow}.

\subsection{Flow of $\bar g$}

We start with the analysis of the sequence $\gbar$, which obeys the recursion
\begin{equation} \label{e:recursion-g2}
  \bar g_{j+1} = \bar g_j - \beta_j \bar g_j^2,
  \quad\quad \gbar_0 = g_0 > 0.
\end{equation}
The following lemma collects the information we need
about $\gbar$.

\begin{lemma} \label{lem:elementary-recursion}
Assume (A1). The following statements hold if $\gbar_0>0$
is sufficiently small, with all constants independent of $\jm$ and $\gbar_0$.
\begin{enumerate}[(i)]
\item
  For all $j$, $\gbar_j > 0$,
  \begin{equation} \lbeq{gbarmon}
    \bar g_j = O(\inf_{k\leq j}\bar g_k), \quad \text{and} \quad
    \bar g_j \bar g_{j+1}^{-1}
    = 1+O(\chi_j \bar g_j)
    = 1+O(\gbar_0).
  \end{equation}
  For all $j$ and $k$, $\gbar_j$ is non-increasing in $\beta_k$.

\item
  (a)
  For real $n \in [1,\infty)$ and $m \in [0,\infty)$,
  there exists $C_{n,m}>0$ such that
  for all $k \ge j \ge 0$,
  \begin{equation}
    \lbeq{chigbd-bis}
    \sum_{l=j}^k \chi_l \bar g_l^n |\log \bar g_l|^m \leq C_{n,m}
    \begin{cases}
      |\log \bar g_{k}|^{m+1} & n = 1
      \\
      \chi_j \bar g_j^{n-1} |\log \bar g_j|^m & n > 1
      .
    \end{cases}
  \end{equation}
  (b)
  For real $n\in [1,\infty)$, there exists $C_n>0$ such that for all $j \geq 0$,
  \begin{equation} \lbeq{chigbd2}
    \chi_j \gbar_j^{n} \leq
    C_n \left(\frac{\gbar_0}{1+\gbar_0 j}\right)^n.
  \end{equation}

\item
  (a) For $\gamma \ge 0$ and $j \geq 0$, there exist constants $c_j = 1+O(\chi_j\gbar_j)$
  (depending on $\gamma$) such that, for all $l \geq j$,
  \begin{equation}
    \lbeq{prodbd}
    \prod_{k=j}^l (1 - \gamma \beta_k \bar g_k)^{-1} =
    \left( \frac{\bar g_{j}}{\bar g_{l+1}} \right)^\gamma (c_j+O(\chi_l \gbar_l)).
  \end{equation}
  The constant $c_j$ is continuous in $g_0$ and if the $\beta_j$ depend continuously on an external
    parameter such that (A1) holds uniformly in that parameter, then $c_j$ is also continuous
    in the external parameter.
  \smallskip

  \noindent
  (b) For $\zeta_j \leq 0$ except for $c^{-1}$ values of $j \leq \jm$, $\zeta_j=O(\chi_j)$, and $j \le l$,
  (with a constant independent of $j$ and $l$),
  \begin{equation}
    \lbeq{zetaprod}
    \prod_{k=j}^l (1-\zeta_k\gbar_k)^{-1} \leq O(1) .
  \end{equation}

\item
  Suppose that $\gbar$ and $\g$ each satisfy \eqref{e:recursion-g2}.
  Let $\delta >0$.
  If $|\g_0-\gbar_0| \leq \delta \g_0$ then  $|\g_j-\gbar_j|\leq \delta \g_j(1+O(\gbar_0))$
  for all $j$.
\end{enumerate}
\end{lemma}

\begin{proof}
(i) By \eqref{e:recursion-g2},
\begin{equation}
\label{e:gbarjj}
  \gbar_{j+1} = \gbar_{j}(1 - \beta_{j}\gbar_{j})
  .
\end{equation}
Since $\beta_j =O(\chi_j)$,
by \refeq{gbarjj} the second statement of  \eqref{e:gbarmon}
is a consequence of the first, so
it suffices to verify the first statement of \eqref{e:gbarmon}.
Assume inductively that $\gbar_j>0$ and that $\bar g_j = O(\inf_{k\leq j}\bar g_k)$.
It is then immediate from \eqref{e:gbarjj} that $\gbar_{j+1} > 0$ if $\gbar_0$ is
sufficiently small depending on $\|\beta\|_\infty$,
and that $\gbar_{j+1} \leq \gbar_j$ if $\beta_j \geq 0$.
By (A1), there are at most $c^{-1}$ values of $j \leq \jm$ for which $\beta_j < 0$.
Therefore, by choosing $\gbar_0$ sufficiently small depending on $\betamax$ and $c$,
it follows that $\gbar_j \leq O(\inf_{k \leq j}\gbar_k)$
for all $j \leq \jm$ with a constant that is independent of $\jm$.

To advance the inductive hypothesis  for $j > \jm$,
we use $1-t \leq e^{-t}$ and
$\sum_{l=\jm}^\infty |\beta_l| \le \sum_{n=1}^\infty \Omega^{-n} =O(1)$ to obtain,
for $j \geq k \geq \jm$,
\begin{equation} \lbeq{gjgkmon}
  \gbar_{j}
  \leq \gbar_{k} \exp\left[-\sum_{l=k}^{j-1} \beta_l \gbar_l\right]
  \leq \gbar_{k} \exp\left[C \gbar_{k} \sum_{l=k}^{j-1} |\beta_l| \right]
  \leq O(\gbar_k).
\end{equation}
This shows that $\gbar_j = O(\inf_{\jm \leq k\leq j}\gbar_{k})$.
However, by the inductive hypothesis,
$\gbar_{\jm} = O(\inf_{k\leq \jm}\gbar_k)$ for $j\leq \jm$, 
and hence for $j > \jm$ we do have $\gbar_j = O(\inf_{k \leq j}\gbar_k)$
as claimed.
This completes the verification of the first bound of \eqref{e:gbarmon} and thus,
as already noted, also of the second.

The monotonicity of $\gbar_j$ in $\beta_k$ can be proved as follows.
Since $\gbar_j$ does not depend on $\beta_k$ if $k \geq j$ by definition,
we may assume that $k<j$. Moreover, by replacing $j$ by $j+k$, we may assume that $k=0$.
Let $\gbar_j' = \partial\gbar_j/\partial\beta_0$. Since $\gbar_0'=0$,
\begin{equation}
  \gbar_{1}' = -\gbar_0^2 < 0.
\end{equation}
Assuming that $\gbar_j' < 0$ by induction, it follows that for $j\geq 1$,
\begin{equation}
  \gbar_{j+1}' = \gbar_j'(1-2\beta_j\gbar_j)
  = \gbar_j'(1+O(g_0)) < 0,
\end{equation}
and the proof of monotonicity is complete.

\smallskip \noindent (ii-a)
We first show that if $\psi: \R_+ \to \R$ is absolutely continuous, then
\begin{equation}
\lbeq{gbarsumbis}
  \sum_{l=j}^{k} \beta_l \psi(\bar g_l) \bar g_l^2
  = \int_{\bar g_{k+1}}^{\bar g_{j}} \psi(t) \; dt
  + O\left(\int_{\bar g_{k+1}}^{\bar g_{j}} t^2 |\psi'(t)| \; dt
  \right)
  .
\end{equation}
To prove \eqref{e:gbarsumbis}, we apply \eqref{e:recursion-g2} to obtain
\begin{equation}
\lbeq{intsum}
  \sum_{l=j}^{k} \beta_{l} \psi(\gbar_{l}) \gbar_l^2
  =
  \sum_{l=j}^{k} \psi(\gbar_{l}) (\gbar_{l} - \gbar_{l+1})
  =
  \sum_{l=j}^{k}
  \int_{\gbar_{l+1}}^{\gbar_{l}} \psi(\gbar_{l}) \,dt
  .
\end{equation}
The integral can be written as
\begin{equation}
\lbeq{psiint}
  \int_{\gbar_{l+1}}^{\gbar_{l}} \psi(\gbar_{l}) \,dt
  =
  \int_{\gbar_{l+1}}^{\gbar_{l}} \psi(t) \,dt
  +
  \int_{\gbar_{l+1}}^{\gbar_{l}} \int_t^{\gbar_l} \psi'(s) \, ds \, dt
  .
\end{equation}
The first term on the right-hand side of
\eqref{e:gbarsumbis} is then the sum over $l$ of the
first term on the right-hand side of \refeq{psiint}, so
it remains to estimate the double integral.
By Fubini's theorem,
\begin{align}
\lbeq{Fub}
  \int_{\gbar_{l+1}}^{\gbar_{l}} \int_t^{\gbar_l} \psi'(s) \, ds \, dt
  &= \int_{\gbar_{l+1}}^{\gbar_{l}} \int_{\gbar_{l+1}}^s \psi'(s) \, dt \, ds
  = \int_{\gbar_{l+1}}^{\gbar_{l}} (s-\gbar_{l+1}) \psi'(s) \, ds
  .
\end{align}
By
\eqref{e:recursion-g2} and \eqref{e:gbarmon},
for $s$ in the domain of integration
we have
\begin{equation}
  |s-\gbar_{l+1}|
  \leq
  |\gbar_l - \gbar_{l+1}|
  =
  |\beta_l| \bar g_l^2
  \leq
  (1+O(\bar g_0)) |\beta_l| \bar g_{l+1}^2
  \leq O(s^2)
  .
\end{equation}
This permits us to estimate \refeq{Fub} and conclude \refeq{gbarsumbis}.

Direct evaluation of the integrals in \eqref{e:gbarsumbis} with $\psi(t) = t^{n-2}|\log t|^m$
gives
\begin{equation}
  \lbeq{betagbd}
  \sum_{l=j}^k \beta_l \bar g_l^n |\log \bar g_l|^m
  \leq C_{n,m}
  \begin{cases}
    |\log \bar g_{k}|^{m+1}  & n = 1
    \\
    \bar g_j^{n-1} |\log \bar g_j|^m  & n > 1
    .
  \end{cases}
\end{equation}
To deduce \eqref{e:chigbd-bis}, we only consider the case $n>1$, as the case $n=1$ is similar.
Suppose first that $j \le \jm$.
Assumption (A1) implies that
\begin{equation}
  1 \leq \frac{\beta_l}{c} + \left(1+\frac{|\beta_l|}{c}\right)1_{\beta_l<c}
  \leq O(\beta_l) + O(1_{\beta_l<c}),
\end{equation}%
and therefore
\begin{align}
  \sum_{l=j}^k \chi_l \bar g_l^n |\log \bar g_l|^m
  & \leq
  \sum_{l=j}^{\jm} O(\beta_l) \bar g_l^n |\log \bar g_l|^m
  +
  \sum_{l=j}^{\jm} O(\1_{\beta_l < c}) \bar g_l^n |\log \bar g_l|^m
  \nnb
  & \quad
  +
  \sum_{l=\jm+1}^{k} \Omega^{-(l-\jm)_+} \bar g_l^n |\log \bar g_l|^m
  .
\end{align}
By \refeq{betagbd}, the first term is bounded by
$O(\bar g_j^{n-1} |\log \bar g_j|^m)$.  The second term obeys the same bound,
by (A1) and \eqref{e:gbarmon}, as does the last term
(which is only present when $\jm<\infty$) due to
the exponential decay.
This proves \refeq{chigbd-bis}
for the case $j \le \jm$.
On the other hand,
if $j>\jm$, then again using the exponential decay of $\chi_l$ and
\eqref{e:gbarmon}, we obtain
\begin{equation}
  \sum_{l=j}^k \chi_l \bar g_l^n |\log \bar g_l|^m
  \leq
  C\chi_j \bar g_{j}^n |\log \bar g_{j}|^m
  \leq
  C\gbar_0\chi_j \bar g_{j}^{n-1} |\log \bar g_{j}|^m
  .
\end{equation}
This completes the proof of \eqref{e:chigbd-bis} for the case $n>1$.

\smallskip \noindent (ii-b)
To prove \refeq{chigbd2}, let $c>0$ be as in Assumption~(A1) and set
$\hat g_{j+1} = \hat g_j - c \hat g_j^2$ with $\hat g_0 = \gbar_0$.
The sequence $(\hat g_j)$
satisfies the bound \refeq{chigbd2}, since application of \refeq{gbarsumbis}
with $\psi(t) = t^{-2}$ gives $(k+1)c=\hat g_{k+1}^{-1}-\hat g_0^{-1}
+ O(\log (\hat g_0/\hat g_{k+1})$ and hence $\hat g_j \sim \hat g_0/(1+c \hat g_0 j)$.
It therefore suffices to prove that $\chi_j \gbar_j^n \le O(\hat g_j^n)$.

We first show that it suffices to prove that
\begin{equation}
  \lbeq{smj}
  \chi_j\gbar_j = \gbar_j \le O(\hat g_j) \quad \text{for all $j \le \jm$}
\end{equation}
(the first inequality holds since $\chi_j=1$ for $j\le \jm$ by definition).
To see this, we note that
for $\gbar_0>0$ sufficiently small and for all $j$,
\begin{equation}
  \Omega^{-1} \leq (1-c\gbar_0)^n \leq (1-c\hat g_j)^n = \left(\frac{\hat g_{j+1}}{\hat g_j}\right)^n.
\end{equation}
For $j>\jm$, by \refeq{smj} and the fact that
$\gbar_j = O(\gbar_{j_\Omega})$ by \refeq{gbarmon},
this implies that
\begin{equation}
  \chi_j \gbar_j^n
  \leq O(\Omega^{-(j-j_\Omega)}\gbar_{j_\Omega}^n)
  \leq O(\Omega^{-(j-j_\Omega)}\hat g_{j_\Omega}^n)
  \leq O \left( \prod_{l=j_{\Omega}}^{j-1} \frac{\hat g_{l+1}}{\hat g_l}  \right)^n
  \hat g_{j_\Omega}^n
  = O(\hat g_j^n)
  .
\end{equation}
For $j\le \jm$, since $\chi_j=1$, it suffices to prove \refeq{chigbd2} with $n=1$, i.e.,
\refeq{smj}.

Let $\tilde\beta_j = \min \{ c, \beta_j \}$, and
define $\tilde g_j$ by the recursion $\tilde g_{j+1} = \tilde g_j - \tilde \beta_j \tilde g_j^2$ with $\tilde g_0 = \gbar_0$.
By the monotonicity in $\beta$ asserted in part (i),
\begin{equation} \lbeq{gtildebds}
  \gbar_j \leq \tilde g_j, \qquad \hat g_j \leq \tilde g_j.
\end{equation}
Denote by $0\leq j_1 < j_2 < \dots < j_m$ the
sequence of $j \le \jm$ such that $\beta_j < c$.  By (A1), the number of
elements in this sequence is indeed finite.
By the first inequality of \refeq{gtildebds} and the definition of $\tilde g_j$,
it follows that, for $j \leq \jm$,
\begin{align}
  \gbar_j \leq \tilde g_j
  = \tilde g_0
  \prod_{l=0}^{j-1} (1-\tilde\beta_l \tilde g_l)
  = \tilde g_0
  \prod_{l=0}^{j-1} (1-c \tilde g_l)
  \prod_{k \le m: j_m \leq j-1} \frac{1-\beta_{j_k} \tilde g_{j_k}}{1-c\tilde g_{j_k}}
  .
\end{align}
Thus, with $\tilde g_0 = \hat g_0$, the second inequality of \refeq{gtildebds},
and the definition of $\hat g_j$,
\begin{equation}
  \gbar_j \leq \hat g_0
  \prod_{l=0}^{j-1} (1-c \hat g_l)
  \prod_{k \le m: j_m \leq j-1} \frac{1-\beta_{j_k} \tilde g_{j_k}}{1-c\tilde g_{j_k}}
  = \hat g_j
  \prod_{k \le m: j_m \leq j-1} \frac{1-\beta_{j_k} \tilde g_{j_k}}{1-c\tilde g_{j_k}}
  .
\end{equation}
The product on the last line is a product
of at most $m$ factors which are each $1+O(g_0)$, and can thus be bounded by $1+O(g_0)$.
In particular,
\begin{equation}
  \lbeq{chigbar1}
  \gbar_{j} \leq (1+O(g_0)) \hat g_{j} \leq O(\hat g_j)
  \quad
  \text{for $j \leq \jm$}
  ,
\end{equation}
and the proof of \refeq{chigbd2} is complete.

\smallskip \noindent (iii-a)
By Taylor's theorem and \eqref{e:recursion-g2}, there
exists $r_k = O(\beta_k\bar g_k)^2$ such that
\begin{equation}
\lbeq{Taygam}
  (1 - \gamma \beta_k \bar g_k)^{-1}
  =
  (1 - \beta_k \bar g_k)^{-\gamma} (1+ r_k)
  =
  \left(\frac{\bar g_{k}}{\bar g_{k+1}}\right)^{\gamma} (1+ r_k).
\end{equation}
For $l \ge j$, let
\begin{equation} \label{e:cjlprod}
  c_{j,l} = \prod_{k=j}^l ( 1 + r_{k} )
  = \exp\left(\sum_{k=j}^l\log( 1 + r_{k} )\right).
\end{equation}
Since $\log(1+r_k) = O(\chi_k\gbar_k^2)$, it follows from
\eqref{e:chigbd-bis} that the sum on the right-hand side
of \refeq{cjlprod} is bounded by $O(\chi_j\gbar_j)$ uniformly in $l$.
We can thus define
\begin{equation}
  c_{j} = \exp\left(\sum_{k=j}^\infty\log( 1 + r_{k} )\right) = 1+O(\chi_j\gbar_j).
\end{equation}
The bound on the sum also shows
\begin{equation}
  c_j - c_{j,l}
  = c_j \left(1- \exp\left(-\sum_{k=l}^\infty\log( 1 + r_{k} )\right)\right)
  = c_j (1+O(\chi_l\gbar_l)).
\end{equation}
Moreover, these estimates hold uniformly in a neighborhood of $g_0$ and in the external parameter,
by assumption. Thus the dominated convergence theorem implies continuity of $c_j$,
both in $g_0$ and in the external parameter, and
the proof is complete.

\smallskip \noindent (iii-b)
Since $\zeta_j \leq 0$ for all but $c^{-1}$ values of $j \leq \jm$,
by \eqref{e:gbarmon} with $\gbar_0$ sufficiently small,
 $\prod_{k=j}^l (1-\zeta_k\bar g_k)^{-1} \leq O(1)$ for $l \leq \jm$,
with a constant independent of $\jm$.
For $j \geq \jm$, we use
$1/(1-x) \leq 2e^{x}$ for $x\in[-\half,\half]$ to obtain
\begin{equation}
  \prod_{k=j}^l (1-\zeta_k\bar g_k)^{-1}
  \leq
  2 \exp\left[\sum_{k=j}^l \zeta_k\bar g_k \right]
  \leq
  2 \exp\left[C\gbar_j \sum_{k=\jm}^\infty \chi_k \right]
  \leq
  O(1).
\end{equation}
The bounds for $l \leq \jm$ and $j \geq \jm$ together imply \refeq{zetaprod}.

\smallskip \noindent (iv)
If $|\g_{j}-\gbar_{j}| \leq \delta_j \g_j$ then by \eqref{e:recursion-g2},
\begin{equation}
  |\g_{j+1}-\gbar_{j+1}|
  =
  |\g_j - \gbar_j|(1-\beta_j(\g_j+\gbar_j))
  \leq
  \delta_{j+1} \g_{j+1}
\end{equation}
with
\begin{equation}
  \delta_{j+1}
  = \delta_j \frac{1-\beta_j(\g_j+\gbar_j)}{1-\beta_j\g_j}
  = \delta_j \left(1-\frac{\beta_j\gbar_j}{1-\beta_j\g_j}\right)
  .
\end{equation}
In particular, if $\beta_j \geq 0$, then $\delta_{j+1}\leq\delta_j$.
By (A1), there are at most $c^{-1}$ values of $j \leq \jm$ for which $\beta_j < 0$,
and hence
$\delta_j \leq \delta (1+O(\gbar_0))$ for $j \leq \jm$.
The desired estimate therefore holds for $j \le \jm$.
For $j \ge l > \jm$,
as in \eqref{e:gjgkmon} we have
\begin{equation}
  \prod_{k=l}^j (1+O(\beta_k\gbar_k))
  \leq \exp\left[O(\gbar_l) \sum_{k=l}^j\chi_k \right] \leq 1+O(\gbar_0),
\end{equation}
and thus the claim remains true also for $j > \jm$.
\end{proof}

\subsection{Flow of $\bar z$ and $\bar \mu$}

We now establish the existence of unique solutions to
the $\bar z$ and $\bar \mu$ recursions  with boundary
conditions $\bar z_\infty = \bar\mu_\infty =0$, and obtain estimates
on these solutions.

\begin{lemma}
\label{lem:greekbds}
  Assume (A1--A2).  If $\gbar_0$ is sufficiently small then there
  exists a unique solution to \refeq{barflow} obeying $\zbar_\infty=\mubar_\infty=0$.
  This solution
  obeys $\zbar_j = O(\chi_j \bar g_j)$ and
  $\bar \mu_j = O(\chi_j \bar g_j)$.
  Furthermore, if the maps $\bar\varphi_j$ depend continuously on $m \in \Mext$
  and (A1--A2) hold with uniform constants, then
  $\gbar_j$, $\zbar_j$ and $\mubar_j$ are continuous in $\Mext$.
\end{lemma}

\begin{proof}
By \refeq{Mmatrix}, $\bar z_{j+1}=\bar z_j - \zeta_j \bar g_j \bar z_j - \theta_j \bar g_j^2$, so that
\begin{equation}
  \bar z_j
  =
  \prod_{k=j}^n (1-\zeta_k\bar g_k)^{-1} \bar z_{n+1} + \sum_{l=j}^n \prod_{k=j}^l (1-\zeta_k\bar g_k)^{-1}  \theta_l \bar g_l^2
  .
\end{equation}
In view of \refeq{zetaprod},
whose assumptions are satisfied by (A2),
the unique solution to the recursion for $\zbar$ which
obeys the boundary condition $\bar z_\infty =0$  is
\begin{equation} \lbeq{zbar}
  \bar z_j
  =
  \sum_{l=j}^\infty \prod_{k=j}^l (1-\zeta_k\bar g_k)^{-1} \theta_l \bar g_l^2
  ,
\end{equation}
and by (A2), \refeq{chigbd-bis}, and \refeq{zetaprod},
\begin{equation} \lbeq{zbarbd}
  |\bar z_j|
  \le \sum_{l=j}^\infty O(\chi_l) \bar g_l^2
  \leq O(\chi_j\bar g_j)
  .
\end{equation}
Since $\gbar_j$ is defined by a
finite recursion, its continuity in $m\in\Mext$ follows from the assumed
continuity of each $\beta_k$ in $m$.
To verify continuity of $\zbar_j$ in $m$,
let $\zbar_{j,n} = \sum_{l=j}^n \prod_{k=j}^l (1-\zeta_k\bar g_k)^{-1} \theta_l \bar g_l^2$.
Since $\gbar_j$ is continuous in $\Mext$ for any $j \geq 0$, $\zbar_{j,n}$ is also continuous, for any $j \leq n$.
By \refeq{zetaprod} and \refeq{chigbd-bis}--\refeq{chigbd2},
$|\zbar_j-\zbar_{j,n}| \leq O(\chi_n \gbar_n) \to 0$ as $n \to \infty$, uniformly in $m$,
and thus, as a uniform limit of continuous functions,  $\zbar_j$ is continuous in $m\in\Mext$.

For $\bar \mu$, we first define
\begin{equation}
  \sigma_j = \eta_j  \gbar_j + \gamma_j \zbar_j
  - \upsilon_j^{gg} \gbar_j^2 - \upsilon_j^{gz} \gbar_j \zbar_j - \upsilon_j^{zz}\zbar_j^2,
  \quad
  \tau_j  = \upsilon_j^{g\mu}\gbar_j + \upsilon_j^{z\mu}\zbar_j ,
\lbeq{etatau}
\end{equation}
so that the recursion for $\mubar$ can be written as
\begin{equation}
    \mubar_{j+1} =(\lambda_j - \tau_j)\mubar_j + \sigma_j.
\end{equation}
Alternatively,
\begin{equation}
\lbeq{mubarrec}
    \mubar_{j} =(\lambda_j - \tau_j)^{-1}(\mubar_{j+1}-\sigma_j).
\end{equation}
Given $\alpha \in (\lambda^{-1},1)$,
we can choose $\gbar_0$ sufficiently small that
\begin{equation}
\lbeq{rhodef}
    \tfrac 12 \lambda^{-1} \le (\lambda_j - \tau_j)^{-1} \le \alpha.
\end{equation}
The limit of repeated iteration of \refeq{mubarrec} gives
\begin{equation}
  \lbeq{mubar}
  \bar \mu_j
  =
  - \sum_{l=j}^\infty
  \left( \prod_{k=j}^l (\lambda_k - \tau_k)^{-1} \right) \sigma_l
\end{equation}
as the unique solution which obeys the boundary condition $\mu_\infty =0$.
Geometric convergence of the sum is guaranteed by \refeq{rhodef},
together with the fact that $\sigma_j \le O(\chi_j\gbar_j) \le O(1)$.
To estimate \refeq{mubar}, we use
\begin{equation}
  |\bar \mu_j|
  \le
  \sum_{l=j}^\infty
  \alpha^{l-j+1}
  O(\chi_l \bar g_l)
  .
\end{equation}
Since $\alpha <1$, the first bound of \eqref{e:gbarmon}
and monotonicity of $\chi$ imply that
\begin{equation} \lbeq{mubar-estimate}
  |\bar \mu_j|
  \le
  O(\chi_j \gbar_j)
  .
\end{equation}
The proof of continuity of $\mubar_j$ in $\Mext$ is analogous to that for $\zbar_j$.
This completes the proof.
\end{proof}

\subsection{Differentiation of quadratic flow}

The following lemma gives estimates on the derivatives of the components
of $\bar V_j$ with respect to the initial condition
$\gbar_0$.
We write $f'$ for the derivative of $f$ with respect to $g_0=\bar g_0$.
These estimates are an ingredient in the proof of Theorem~\ref{thm:flow}(ii).

\begin{lemma}
\label{lem:barflowderiv}
For each $j \ge 0$, $\bar V_j=(\gbar_j$, $\zbar_j$,  $\mubar_j)$
is twice differentiable with respect to the initial condition $\gbar_0 > 0$,
and the derivatives obey
\begin{alignat}{6} \lbeq{xbar-derivatives}
  \gbar_{j}'
  &=
  O\left( \frac{\bar g_{j}^2}{\bar g_{0}^2} \right)
  ,
  &\qquad
  \zbar_j'
  &=
  O\left( \chi_j \frac{\bar g_{j}^2}{\bar g_{0}^2} \right)
  ,
  &\qquad
  \mubar_j'
  &=
  O\left( \chi_j \frac{\bar g_{j}^2}{\bar g_{0}^2} \right)
  ,
  \\
  \lbeq{xbar-derivatives2}
  \gbar_{j}''
  &=
  O\left( \frac{\bar g_{j}^2}{\bar g_{0}^3} \right)
  ,
  &\qquad
  \zbar_j''
  &=
  O\left( \chi_j \frac{\bar g_{j}^2}{\bar g_{0}^3} \right)
  &,
  \qquad
  \mubar_j''
  &=
  O\left( \chi_j \frac{\bar g_{j}^2}{\bar g_{0}^3} \right)
  .
\end{alignat}
\end{lemma}

\begin{proof}
Differentiation of \refeq{gbar} gives
\begin{equation}
  \gbar_{j+1}' = \gbar_j'(1 - 2\beta_j \gbar_j)
  ,
\end{equation}
from which we conclude by iteration and $\gbar_0'=1$ that for $j \ge 1$,
\begin{equation}
  \gbar_{j}' = \prod_{l=0}^{j-1}(1 - 2\beta_l \gbar_l).
\end{equation}
Therefore, by \refeq{prodbd},
\begin{equation} \lbeq{gbarprime}
  \gbar_{j}'
  =
  \left( \frac{\bar g_{j}}{\bar g_{0}} \right)^2 (1+O(\bar g_0)).
\end{equation}
For the second derivative, we use $\gbar_0''=0$ and
$\gbar_{j+1}'' = \gbar_j''(1-2\beta_j\gbar_j)-2\beta_j\gbar_j'^2$ to
obtain
\begin{equation}
  \gbar_j''
  = -2 \sum_{l=0}^{j-1} \beta_l \gbar_l'^2 \prod_{k=l}^{j-2} (1-2\beta_k\gbar_k).
\end{equation}
With the bounds of Lemma~\ref{lem:elementary-recursion},
this gives
\begin{equation}
  \gbar_j''
  = O\left(\frac{\gbar_{j}}{\gbar_0}\right)^2 \sum_{l=0}^{j-1} \beta_l \frac{\gbar_l^2}{\gbar_0^2}
  = O\left(\frac{\gbar_j^2}{\gbar_0^3}\right)
  .
\end{equation}

For $\zbar$, we
define $\sigma_{j,l} = \prod_{k=j}^l (1-\zeta_k \gbar_k)^{-1}$.  Then
\refeq{zbar} becomes $\zbar_j = \sum_{l=j}^\infty \sigma_{j,l}\theta_l\gbar_l^2$.
By \refeq{zetaprod}, $\sigma_{j,l}=O(1)$.
It then follows from (A2), \refeq{gbarprime}, and
Lemma~\ref{lem:elementary-recursion}(ii,iii-b)
that
\begin{equation}
  \sigma_{j,l}'
  = \sigma_{j,l} \sum_{k=j}^l (1-\zeta_k\gbar_k)^{-1} \zeta_k \gbar_k'
  =  \sum_{k=j}^l O(\zeta_k\gbar_k')
  = O\left(\chi_j \frac{\bar g_j}{\gbar_0^2}\right).
\end{equation}
We differentiate \refeq{zbar} and apply
\refeq{gbarprime} and
Lemma~\ref{lem:elementary-recursion}(ii)
to obtain
\begin{equation} \lbeq{zbarprime}
  \zbar_j'
  =
  \sum_{l=j}^\infty \sigma_{j,l}' \theta_l \gbar_l^2
  +
  2 \sum_{l=j}^\infty \sigma_{j,l} \theta_l \gbar_l \gbar_l'
  =
  O\left( \chi_j \frac{\bar g_{j}^2}{\bar g_{0}^2} \right).
\end{equation}
Similarly, $\sigma_{j,l}'' = O(\gbar_j/\gbar_0^3)$ and
\begin{align} \lbeq{zbarprimeprime}
  \zbar_j''
  &=
  \sum_{l=j}^\infty \sigma_{j,l}'' \theta_l \gbar_l^2
  +
  4 \sum_{l=j}^\infty \sigma_{j,l}' \theta_l \gbar_l \gbar_l'
  +
  2\sum_{l=j}^\infty \sigma_{j,l} \theta_l (\gbar_l\gbar_l''+\gbar_l'^2)
  = O\left(\chi_j\frac{\gbar_j^2}{\gbar_0^3}\right)
\end{align}
using the fact
that $\gbar_j^3/\gbar_0^4 = O(\gbar_j^2/\gbar_0^3)$ by \eqref{e:gbarmon}.
It is straightforward to justify the differentiation under the sum in \eqref{e:zbarprime}--\eqref{e:zbarprimeprime}.

For $\mubar_j$,
we recall from \refeq{rhodef}--\refeq{mubar} that
\begin{equation} \lbeq{mubar-diff}
  \bar \mu_j
  =
  - \sum_{l=j}^\infty
  \left( \prod_{k=j}^l (\lambda_k - \tau_k)^{-1} \right) \sigma_l
  ,
\end{equation}
with $\tau_j$ and $\sigma_l$ given by \refeq{etatau}, and with
$0 \le (\lambda_j - \tau_j)^{-1} \le \alpha<1$.  This gives
\begin{equation}
  \mubar_j'
  =
  - \sum_{l=j}^\infty
  \left( \prod_{k=j}^l (\lambda_k - \tau_k)^{-1} \right)
  \left( \sigma_l' + \sum_{i=j}^l (\lambda_i - \tau_i)^{-1}\tau_i' \right).
\end{equation}
The first product is bounded by $\alpha^{l-j+1}$, and this
exponential decay, together with \refeq{etatau}, \refeq{zbarbd}, and the bounds just
proved for $\gbar'$ and $\zbar'$, lead to the upper bound
$|\mubar_j'| \le O(\chi_j \gbar_j^2 \gbar_0^{-2})$ claimed
in \refeq{xbar-derivatives}.  Straightforward further calculation
leads to the bound on $\mubar_j''$ claimed
in \refeq{xbar-derivatives2} (the leading behaviour can be seen from
the $\zbar_j''$ contribution to the $\sigma_l''$ term).
\end{proof}

\subsection{Proof of Proposition~\ref{prop:approximate-flow}}

\begin{proof}[Proof of Proposition~\ref{prop:approximate-flow}]
  The estimates
  \refeq{approximate-flow} follow  from Lemma~\ref{lem:elementary-recursion}(ii)
  and Lemma~\ref{lem:greekbds}.
  The continuity of $\gbar_j$, $\zbar_j$ and $\mubar_j$ in $m$
  follows from Lemma~\ref{lem:greekbds}, and their differentiability in
  the initial condition $\gbar_0$ follows from Lemma~\ref{lem:barflowderiv}.
\end{proof}

\section{Proof of main result}
\label{sec:pf}

In this section, we prove Theorem~\ref{thm:flow}.
We begin in Section~\ref{sec:pert-ODE} with a sketch of the main
ideas, without entering into details.

\subsection{Proof strategy}
\label{sec:pert-ODE}

Two difficulties in proving Theorem~\ref{thm:flow} are:
(i) from the point of view of dynamical systems, the evolution map $\Phi$ is
not hyperbolic; and
(ii) from the point of view of nonlinear differential equations,
a priori bounds that any solution to \eqref{e:flow} must satisfy are not
readily available due to the presence of both initial and final
boundary conditions.

Our strategy is to consider the one-parameter
family of evolution maps $(\Phi^t)_{t\in[0,1]}$ defined by
\begin{equation} \lbeq{Phit}
  \Phi^t(x) = \Phi(t,x) = (\psi(x), \bar \varphi(x) + t \rho(x))
  \quad \text{for $t \in [0,1]$,}
\end{equation}
with the $t$-independent boundary conditions that $(K_0,g_0)$ is given
and that $(z_\infty,\mu_\infty)=(0,0)$.
This family interpolates between the
problem $\Phi^1=\Phi$
we are interested in, and the simpler problem
$\Phi^0=\bar\Phi= (\psi , \bar \varphi )$. The unique solution for $\bar\Phi$
is $\bar x_j = (\bar K_j, \bar V_j)$, where $\bar V$ is the unique solution
of $\bar\varphi$ from Section~\ref{sec:approximate-flow},
and where $\bar K_j$ is defined inductively for $j \ge 0$
(recall Lemma~\ref{lem:small-ball}) by
\begin{equation} \lbeq{Kbar-def}
  \bar K_{j+1} = \psi_j(\bar V_j, \bar K_j),
  \qquad \bar K_0 = K_0.
\end{equation}
We refer to $\bar x$ as the \emph{approximate flow}.

We seek a $t$-dependent global flow $x$ which obeys
the generalisation of \refeq{flow} given by
\begin{equation} \lbeq{flowt}
  x_{j+1} = \Phi^t_j(x_j).
\end{equation}
Assuming that $x_j = x_j(t)$ is differentiable in $t$ for each
$j \in \N_0$, we set
\begin{equation} \lbeq{xdot}
  \dot x_j = \ddp{}{t} x_j
  .
\end{equation}
Then differentiation of \eqref{e:flowt}
shows that a family of flows $x = (x_j(t))_{j\in\N_0, t\in[0,1]}$
must satisfy the infinite nonlinear system of ODEs
\begin{equation} \lbeq{ODE}
  \dot x_{j+1} - D_x\Phi_j(t, x_j)\dot x_j = \rho_j(x_j),
  \quad x_j(0) = \bar x_j.
\end{equation}
Conversely, any solution $x(t)$ to \eqref{e:ODE}, for which each $x_j$ is
continuously differentiable in $t$,
gives a global flow for each $\Phi^t$.

We claim that \eqref{e:ODE} can be reformulated
as a well-posed nonlinear ODE
\begin{equation} \lbeq{ODE-F}
  \dot x = F(t,x), \quad x(0)= \bar x,
\end{equation}
in a Banach space of sequences $x = (x_0, x_1, \dots)$
with carefully chosen weights,
and for a suitable
nonlinear functional $F$.
For this, we consider the \emph{linear} equation
\begin{equation}
  y_{j+1} - D_x\Phi_j(t,x_j)y_j = r_j,
\end{equation}
where the sequences $x$ and $r$ are held fixed.
Its solution with the same boundary conditions as stated below \eqref{e:Phit} is written as $y=S(t,x)r$.
Then we define $F$,
which we consider as a map on sequences, by
\begin{equation} \lbeq{F-def1}
  F(t,x) = S(t,x)\rho(x).
\end{equation}
Thus $y=F(t,x)$ obeys the equation
$y_{j+1} - D_x\Phi_j(t,x_j)y_j = \rho_j(x)$, and hence \refeq{ODE-F}
is equivalent to \refeq{ODE} with the same boundary conditions.

The main work in the proof is to obtain good estimates for $S(t,x)$,
in the Banach space of weighted sequences,
which allow us to treat \eqref{e:ODE-F} by the standard theory of ODE.
We establish bounds on the solution simultaneously with existence,
via the weights in the norm. These weights are useful to obtain bounds on the
solution, but they are also essential in the formulation of
the problem as a well-posed ODE.

As we show in more detail in Section~\ref{sec:step1}
below,
the occurrence of $D_x\Phi_j(t, x_j)$ in \refeq{ODE}, rather
than the naive linearisation $D_x \Phi_j(0)$ at the fixed point $x=0$, replaces the
eigenvalue $1$ in the upper left corner of the square matrix
in \refeq{Mmatrix}
by the eigenvalue
$1 - 2\beta_jg_j$, which is less than $1$ except for those negligible $j$ values
for which $\beta_j<0$.
This helps address difficulty (i) mentioned above.
Also, the
weights guarantee that a solution in the Banach space obeys the final conditions
$(z_\infty,\mu_\infty) =(0,0)$, thereby helping to solve difficulty (ii).

\subsection{Sequence spaces and weights}

We now introduce the Banach spaces of sequences used in the reformulation
of \eqref{e:ODE} as an ODE. These are weighted $l^\infty$-spaces.

\begin{defn} \label{defn:norms}
  Let $X^*$ be the space of sequences $x=(x_j)_{j\in\N_0}$ with $x_j \in X_j$.
  For each $\alpha=K,g,z,\mu$ and $j \in \N_0$, we fix a positive weight
  $w_{\alpha,j} > 0$.  We write $x_j \in X_j= \Wcal_j \oplus \Vcal$ as
  $x_j = (x_{\alpha,j})_{\alpha =K, g,z,\mu}$.
  Let
  \begin{equation} \lbeq{norms}
    \|x_j\|_{X^w_j}
    =
    \max_{\alpha=K,g,z,\mu} \;
    (w_{\alpha,j})^{-1} \|x_{\alpha,j}\|_{X_j}
    ,
    \quad
    \|x\|_{X^{w}}
    =
    \sup_{j \in \N_0} \|x_j\|_{X^w_j},
  \end{equation}
  and
  \begin{equation}
    X^w = \{ x \in X^*: \|x\|_{X^w} < \infty \}.
  \end{equation}
\end{defn}

It is not difficult to check that $X^w$
is a Banach space for any positive weight sequence $w$.
The required weights are
defined in terms of the sequence $\g = (\g_j)_{j\in\N_0}$ which is the
same as the sequence $\bar g$ for a \emph{fixed} $\g_0$; i.e., given
$\g_0>0$, it satisfies $\g_{j+1} = \g_j - \beta_j \g_j^2$.
We need two different choices of weights $w$, defined in terms
of the parameters $\domK,\domV$ of \refeq{Djdef} and the parameter
$\domK^*$ of Lemma~\ref{lem:small-ball}.  These are
the weights $\w=\w(\g_0,\domK,\domK^*,\domV)$ and $\v=\v(\g_0,\domK,\domK^*,\domV)$ defined by
\begin{align}
  \lbeq{weights}
  \w_{\alpha,j}  &=
  \begin{cases}
    \hK \chi_j \g_j^3 & \alpha=K
    \\
    \hV\g_j^2 |\log \g_j| & \alpha=g
    \\
    \hV \chi_j \g_j^2 |\log \g_j| & \alpha=z,\mu,
  \end{cases}
  \;\;
  \v_{\alpha,j} =
  \begin{cases}
    \hK\chi_j \g_j^3  & \alpha=K
    \\
    \hV \chi_j \g_j^3 & \alpha=g, z,\mu,
  \end{cases}
\end{align}
where $(\chi_j)$ is the $\Omega$-dependent sequence defined by \refeq{chidef}.
Furthermore, let
$\bar x = (\bar K, \bar V) = \bar x (K_0, g_0)$
denote the sequence in $X^*$ that is uniquely determined
from the boundary conditions $(\bar K_0, \gbar_0) = (K_0,g_0)$ and
$(\zbar_\infty ,\mubar_\infty)=(0,0)$
via $\bar V_{j+1} = \bar\varphi_j(\bar V_j)$
and $\bar K_{j+1} = \psi_j(\bar K_j,\bar V_j)$,
whenever the latter is well-defined.
Given an initial condition $(\smash{\Kr_0}, \g_0)$, let $\x = \bar x(\smash{\Kr_0}, \g_0)$.

Let $s\B$ denote the closed ball of radius $s$ in $X^\w$.
If $\g_0=g_0$ and $\Kr_0=K_0$, the
desired bounds \refeq{VVbar1}--\refeq{VVbar4} are equivalent
to $x \in \x+\b \B$.  Also,
the projection of $\x+\B$ onto the
the $j^{\rm th}$ sequence element is contained in the domain $D_j$
defined by \refeq{Djdef}.
We  always assume that $\g_0$ is close to $g_0 = \bar g_0$,
but not necessarily that they are equal.
The use of $\g$ rather than $\gbar$ permits us to vary the initial condition
$g_0=\gbar_0$ without changing the Banach spaces $X^{\w},X^{\v}$.
The use of $g_0$-dependent weights rather than, e.g., the weight $j^{-2} \log j$ for $\jm=\infty$
(see Remark~\ref{rk:flow}(i))
allows us to obtain estimates with good behaviour as $g_0 \to 0$.
Note that the weight $\w_{g,j}$ does not include a factor $\chi_j$,
and thus does not go to $0$ when $\jm < \infty$
(see Example~\ref{ex:chi}(ii)).

\begin{rk} \label{rk:weights}
  The weights $\w$ apply to the sequence $\dot x$ of \refeq{xdot}.
  As motivation for their definition, consider the explicit example of
  $\rho_j(x_j) = \chi_j g_j^3$.  In this case, the $g$ equation becomes simply
  \begin{equation}
    g_{j+1} = g_j - \beta_j g_j^2 + t\chi_jg_j^3.
  \end{equation}
  With the notation $\dot g_j = \frac{\partial}{\partial t} g_j^t$, differentiation gives
  \begin{equation}
    \dot g_{j+1} = \dot g_j  (1- 2\beta_j g_j + 3t\chi_jg_j^{2}) + \chi_j g_j^3.
  \end{equation}
  Thus, by iteration, using $\dot g_0 =0$, we obtain
  \begin{equation}
    \dot g_{j}
    =
    \sum_{l=0}^{j-1}  \chi_lg_{l}^3 \prod_{k=l+1}^{j-1} (1- 2\beta_k g_k + 3t\chi_kg_k^3).
  \end{equation}
  For simplicity, consider the case $t=0$, for which
  $g=\gbar$.  In this case, it follows from \eqref{e:prodbd},
  \refeq{gbarmon}, and \refeq{chigbd-bis} that
  \begin{align}
    \dot g_{j}
    &
    \le O(1) \sum_{l=0}^{j-1} \left( \frac{\bar g_{j}}{\bar g_{l+1}} \right)^2
    \chi_l \gbar_{l}^3
    = O(1) g_{j}^2 \sum_{l=0}^{j-1}  \chi_l \gbar_{l}
    \le O(1) \gbar_j^2 |\log \gbar_j|,
  \end{align}
  which produces the weight $\w_{g,j}$ of \refeq{weights}.
  (It can be verified using \refeq{gbarsumbis} that if we replace
   $\chi_j$ by $\beta_j$ in the above then
  no smaller weight will work.)
\end{rk}

\subsection{Implications of Assumption~(A3)}
\label{sec:R-bounds}

For $\phi$ equal to either of the maps $\rho,\psi$
of \eqref{e:F}, we define $\phi: \x + \u\B \to X^*$ by
\begin{equation} \lbeq{rho-def}
  (\phi(x))_0 = 0, \quad (\phi(x))_{j+1} = \phi_j(x_j).
\end{equation}
The next lemma expresses immediate
consequences of Assumption~(A3) for $\rho$ and $\psi$
in terms of the weighted spaces.

\begin{lemma} \label{lem:R-bounds}
  Assume (A3) and that $\g_0 >0$ is sufficiently small.
  The map $\rho$ obeys
  \begin{equation}
  \lbeq{rhoA3weight}
    \|\rho(x)\|_{X^\v}
    \leq M \hV^{-1}.
  \end{equation}
  Let $\omega > \kappa\Omega$, and let $\phi$ denote either $\psi$ or $\rho$.
  The map $\phi: \x + \u\B \to X^\v$ 
  is twice continuously
  Fr\'echet differentiable,
  and there exists a constant $C = C(\domK,\domK^*,\domV)$ such that
  \begin{gather}
    \|D_K\rho(x)\|_{L(X^\w,X^\v)}
    \leq C,
    \quad
    \|D_K\psi(x)\|_{L(X^\w,X^\v)}
    \leq \omega,
    \nnb
    \|D_V\phi(x)\|_{L(X^\w,X^\v)}
    \leq O(\g_0 |\log \g_0|),
    \quad
    \|D_x^2\phi(x)\|_{L^2(X^\w,X^\v)}
    \leq C.
    \lbeq{Rw1-bound}
  \end{gather}
\end{lemma}

\begin{proof}
The bound \refeq{rhoA3weight}
is equivalent to \refeq{R-bound} (recall  \refeq{rho-def}) since
\begin{equation} \lbeq{rhowbd}
  \|\rho_{j}(x_j)\|_{X_{j+1}^\v}
  = \v_{g,j+1}^{-1} \|\rho_{j}(x_j)\|_{\Vcal}
  \le \v_{g,j+1}^{-1} M \chi_{j+1} \g_{j+1}^3
  = M/\hV.
\end{equation}

Next we verify
the bounds on the first derivatives in \eqref{e:Rw1-bound}.
By assumptions \eqref{e:R-Lip-bound-K}--\eqref{e:R-Lip-bound-V},
together with \refeq{gbarmon}, the definition of
the weights \eqref{e:weights}, and for \refeq{DRKK-w-bound} also
the fact that  $\chi_{j}/\chi_{j+1} \le \Omega$  by \refeq{chidef},
we obtain for $x \in \x+\u\B$,
\begin{alignat}{3}
  \|D_V\psi_{j}(x_j)\|_{L(X_j^\w,X_{j+1}^\v)}
  &\leq M\chi_{j+1} \g_{j+1}^2 \v_{K,j+1}^{-1} \w_{g,j}
  &&\leq O(\g_0|\log \g_0|),
  \lbeq{DRKV-w-bound}
  \\
  \|D_K\psi_{j}(x_j)\|_{L(X^\w_j, X^\v_{j+1})}
  &\leq \kappa \v_{K,j+1}^{-1} \w_{K,j}
  &&\leq \kappa\Omega (1+O(\g_0)),
  \lbeq{DRKK-w-bound}
  \\
  \|D_V\rho_{j}(x_j)\|_{L(X_j^\w,X_{j+1}^\v)}
  &\leq M\chi_{j+1} \g_{j+1}^2 \v_{g,j+1}^{-1} \w_{g,j}
  &&\leq O(\g_0|\log \g_0|),
  \lbeq{DRVV-w-bound}
  \\
  \|D_K\rho_{j}(x_j)\|_{L(X_j^\w,X_{j+1}^\v)}
  &\leq M\v_{g,j+1}^{-1} \w_{K,j}
  &&
  \leq O(1)
  ,
  \lbeq{DRVK-w-bound}
\end{alignat}
which establishes the bounds on the first derivatives in
\eqref{e:Rw1-bound}, for $\g_0$ sufficiently small.

The bounds on the second derivatives are also immediate consequences of Assumption~(A3).
First, \eqref{e:R-Lip-bound-higher} and the definition of the weights \eqref{e:weights} imply that, for $2 \leq n+m \leq 3$,
\begin{equation}
  \|D_K^nD_V^m\phi\|_{L^{n+m}(X^\w,X^\v)}
  \leq C
  .
\end{equation}
In addition, these bounds on the second and third derivatives
imply that
\begin{align}
  \|\phi(x+y)-\phi(x)-D\phi(x)y\|_{X^\v} & \leq C \|y\|_{X^\w}^2,
  \\
  \lbeq{DDR-bound}
  \|D\phi(x+y) - D\phi(x) - D^2 \phi(x)y\|_{L(X^\w,X^\v)}
  &\leq C \|y\|_{X^\w}^2,
\end{align}
and hence that $\phi: \x + \u\B \to X^\v$ is indeed twice Fr\'echet differentiable.
The above bound on the third derivatives also implies continuity of this differentiability.
This completes the proof.
\end{proof}

The smoothness of $\bar x$ is addressed in the following lemma.

\begin{lemma} \label{lem:xbarderiv}
  Assume (A1--A3), and
  let $\delta >0$ and $\g_0 > 0$ both be sufficiently small.
  Then there exists a neighbourhood $\bar\I = \bar\I_{\delta} \subset \Wcal_0\oplus \R_+$
  of $(\Kr_0,\g_0)$ such that
  $\bar x: \bar\I \to \x + \delta\B$ is continuously Fr\'echet differentiable,
  and
  \begin{equation} \lbeq{barx-dg0-bd}
    \|D_{g_0} \bar x\|_{X^\w}
    \leq O(\g_0^{-2} |\log \g_0|^{-1})
    .
  \end{equation}
  The neighbourhood $\bar\I$ contains a ball centred at $(\Kr_0,\g_0)$
  with radius depending only on $\g_0,\delta$, and the constants in (A1--A3),
  which is bounded below away from $0$, uniformly on compact subsets of small $\g_0>0$.
\end{lemma}

\begin{proof}
Let
\begin{equation}
  \bar\I = ([\half \g_0, 2\g_0]
  \times \Wcal_0) \cap \bar x^{-1}(\x+\delta\B).
\end{equation}
We will show that $\bar\I$ is a neighbourhood of $(\Kr_0,\g_0)$
and that
$\bar x : \bar\I \to \x + \delta\B$ is continuously Fr\'echet differentiable.
Since $\bar x^{-1}(\x+\delta\B) = \bar V^{-1}(\x+\delta\B) \cap \bar K^{-1}(\x+\delta\B)$,
it suffices to show that each of
$\bar V^{-1}(\x+\delta\B)$ and $\bar K^{-1}(\x+\delta\B)$ is a neighbourhood
of $(\Kr_0, \g_0)$, and that each of $\bar V$ and $\bar K$ is continuously
Fr\'echet differentiable on $\bar\I$ as maps with values in subspaces of $X^\w$.

We begin with $\bar V$.
Let $\bar V_j'$ denote the derivative of $\bar V_j$ with respect to
$g_0$, and let $\bar V'=(\bar V'_j)$ denote the sequence of derivatives.
It is straightforward to conclude from
Lemma~\ref{lem:barflowderiv}, Lemma~\ref{lem:elementary-recursion}(iv),
and \refeq{weights} that
\begin{equation}
  \lbeq{DgVbar-bd}
  \|\bar V'\|_{X^\w} \leq O(\g_0^{-2}|\log\g_0|^{-1}).
\end{equation}
In particular, this implies that $\bar V^{-1}(\x+\delta\B)$ contains a
neighbourhood of $\g_0$
satisfying the condition stated below \eqref{e:barx-dg0-bd}.
That $\bar V'$ is actually the derivative of $\bar V$ in the space $X^\w$
can be deduced from the fact that
the sequence $\bar V''(g_0)$ is uniformly bounded in $X^\w$ for $g_0 \in \bar\I_g$
(though not uniform in $\g_0$) by Lemma~\ref{lem:barflowderiv}, using
\begin{equation} \lbeq{VFrechet}
  \|\bar V_j(g_0+\varepsilon)-\bar V_j(g_0) - \varepsilon \bar V_j'(g_0)\|_{X^\w_j} \leq O(\varepsilon^2) \sup_{0<\varepsilon'<\varepsilon} \|\bar V_j''(g+\varepsilon')\|_{X_j^\w}
 .
\end{equation}
The continuity of  $\bar V'$ in $X^\w$ follows similarly.

For $\bar K$, we first note that
$\|D_{K_0}\bar K_0\|_{L(\Wcal_0,\Wcal_0)} =1$, $\|D_{g_0}\bar K_0\|_{\Wcal_0} = 0$.
By (A3) and induction,
\begin{align}
  \|D_{K_0} \bar K_{j+1}\|_{L(\Wcal_0,\Wcal_{j+1})}
  &\leq \kappa \|D_{K_0}\bar K_j\|_{L(\Wcal_0,\Wcal_{j})}
  \leq \kappa^{j+1}
  .
  \lbeq{DKKbar-bd}
\end{align}
Since $\kappa< \Omega^{-1} < 1$, and since $\g_{j+1}/\g_j \to 1$ by \eqref{e:gbarmon}, we obtain
\begin{align}
  \|D_{K_0} \bar K_{j+1}\|_{L(\Wcal_0,\Wcal_{j+1})}
  \leq O (\g_0^{-3}\w_{K,j+1})
  .
  \lbeq{DKKbar-bd-w}
\end{align}
Similarly, by \eqref{e:R-Lip-bound-V} and Lemma~\ref{lem:barflowderiv},
\begin{align}
  \|D_{g_0} \bar K_{j+1}\|_{\Wcal_{j+1}}
  &\leq \kappa \|D_{g_0}\bar K_j\|_{\Wcal_{j}}
       + O(\chi_j \bar g_j^2) \|D_{g_0}\bar V_j\|_{\Vcal}
  \nnb
  &\leq \kappa \|D_{g_0}\bar K_j\|_{\Wcal_{j}}
       + O(\chi_j \bar g_j^4/\gbar_0^2)
  .
  \lbeq{DgKbar-bd-2}
\end{align}
By induction,
again using $\kappa < \Omega^{-1}$ and $\gbar_j \leq \gbar_0$, we conclude
\begin{align}
  \|D_{g_0} \bar K_{j+1}\|_{\Wcal_{j+1}}
  &
  \leq O(\chi_j \bar g_j^4/\bar g_0^2) \leq O(\g_0^{-1} \w_{K,j+1}).
  \lbeq{DgKbar-bd}
\end{align}
These bounds imply
that $\bar K^{-1}(\x+\delta\B)$ contains a neighbourhood of $(\Kr_0,\g_0)$
satisfying the condition stated below \eqref{e:barx-dg0-bd},
and also that
the component-wise derivatives of $\bar K$ with respect to $g_0$ and $K_0$
are respectively
in $X^\w \cong L(\R, X^\w)$ and $L(\Wcal_0, X^\w)$.

To verify that the component-wise
derivative of the sequence $\bar K$ is the Fr\'echet derivative in the sequence space $X^\w$, it again suffices to obtain bounds on
the second derivatives in $X^\w$, as in \eqref{e:VFrechet}.
For example, since $D_{K_0}^2 \bar K_0 = 0$, $D_{K_0} \bar V_j = 0$,  and
\begin{equation} \lbeq{DK2barK-rec}
  D_{K_0}^2 \bar K_{j+1} = D_{K} \psi(\bar K_j, \bar V_j) D_{K_0}^2 \bar K_j
  + D_{K}^2 \psi(\bar K_j, \bar V_j) D_{K_0}\bar K_j D_{K_0}\bar K_j,
\end{equation}
it follows from \eqref{e:DKKbar-bd},
\eqref{e:R-Lip-bound-K}--\eqref{e:R-Lip-bound-higher},
and induction that,
for $(K_0,g_0) \in \bar \I$ with $\bar \I \subset \Wcal_0 \oplus \R$ chosen sufficiently small,
and with $\omega \in (\kappa\Omega,1)$,
\begin{equation}
  \|D_{K_0}^2 \bar K_{j+1}\| \leq
  \kappa \|D_{K_0}^2 \bar K_j\| + O(\gbar_0^{-3} \kappa^{j}\omega^j)
  \leq O (\g_0^{-6}\w_{K,j+1})
  ,
\end{equation}
and thus that the component-wise derivative $D_{K_0}^2 \bar K$ is bounded in the norm of $L^2(\Wcal_0, X^\w)$
for $(K_0,g_0) \in \bar \I$. Similarly, slightly more complicated recursion
relations than \eqref{e:DK2barK-rec}
for $D_{g_0}^2\bar K_j$ and $D_{g_0}D_{K_0}\bar K_j$ show that the component-wise second derivative
of $\bar K$ is uniformly bounded in $L^2(\Wcal_0 \oplus \R, X^\w)$ when $\bar \I$ is again chosen sufficiently small.
This shows as in \eqref{e:VFrechet} that $\bar K$ is continuously Fr\'echet differentiable
from $\bar \I$ to $X^\w$.

We have thus shown that $\bar x$  is continuously Fr\'echet differentiable
from a neighbourhood $\bar \I$ of $(\Kr_0,\g_0)$ to $X^\w$, and \refeq{barx-dg0-bd}
follows from \refeq{DgVbar-bd} and \refeq{DgKbar-bd}.
\end{proof}

\subsection{Reduction to a linear equation with nonlinear perturbation}
\label{sec:odepf}

For given sequences $x,r \in X^*$,
we now consider the equation
\begin{equation} \lbeq{DPhiyr}
  y_{j+1} - D_x\Phi_j(t, x_j)y_j = r_j.
\end{equation}
With $x$ and $r$ fixed, this is an inhomogeneous linear equation in $y$.
Lemma~\ref{lem:DPhiyr} below, which lies at the heart of the proof of
Theorem~\ref{thm:flow}, obtains bounds on solutions to \refeq{DPhiyr}, including
bounds on its $x$-dependence.
The latter allows us
to use the standard theory of ODE in Banach spaces to treat the
original nonlinear equation, where $x$ and $r$ are both functionals of the solution $y$,
as a perturbation of the linear equation.

In addition to the decomposition $X_j=\Wcal_j \oplus \Vcal$,
with $x_j \in X_j$ written $x_j =(K_j,V_j)$,
it is convenient to also use the decomposition $X_j = E_{j} \oplus F_{j}$ with
$E_{j} = \Wcal_j \oplus \R$ and $F_{j}= \R \oplus \R$,
for which we write $x_j=(u_j,v_j)$ with $u_j=(K_j,g_j)$ and $v_j=(z_j,\mu_j)$.
We denote by $\proj_\alpha$ the projection operator onto
the $\alpha$-component of the space in which it is applied,
where $\alpha$ can be in any of $\{K, V\}$, $\{u, v\}=\{(K,g),(z,\mu)\}$,
or $\{K,g,z,\mu\}$.

Recall that the spaces of sequences $X^w$
are defined in Definition~\ref{defn:norms} and the
specific weights $\w$ and $\v$ in \eqref{e:weights}.
The following lemma is proved in Section~\ref{sec:pf-DPhiyr}.

\begin{lemma} \label{lem:DPhiyr}
  Assume (A1--A3).
  Then there is a constant $C_S$, independent of the parameters $\domK$ and $\domV$ in \refeq{Djdef},
  and a constant $C_S'= C_S'(\domK,\domV)$,
  such that if $\g_0 >0$ is sufficiently small, the following hold
  for all $t \in [0,1]$, $x \in \x + \u\B$.
  \begin{enumerate}[(i)]
  \item
  For $r \in X^{\v}$, there exists a
  unique solution $y=S(t,x)r \in X^\w$
  of \eqref{e:DPhiyr} with boundary conditions
    $\proj_{u} y_0 = 0$,  $\proj_{v} y_\infty = 0$.
  \item
  The linear solution operator $S(t,x)$
  satisfies
  \begin{equation} \lbeq{Stx-bound}
    \|S(t,x)
    \|_{L(X^{\v},X^{\w})}
    \leq C_S
    .
  \end{equation}
  \item
  As a map $S : [0,1] \times (\x + \u\B) \to L(X^{\w}, X^{\v})$, the solution operator is continuously
  Fr\'echet differentiable and satisfies
  \begin{equation} \lbeq{DStx-bound}
    \|D_xS(t,x)\|_{L(X^{\w},L(X^{\v},X^{\w}))}
    \leq C_{S}'
    .
  \end{equation}
\end{enumerate}
\end{lemma}

Lemma~\ref{lem:DPhiyr} is supplemented
with the information about the perturbation $\rho$ given by
Lemma~\ref{lem:R-bounds}, and by the
information about the initial condition $\bar x$
provided by Lemmas~\ref{lem:greekbds}.
(Note that the sequence $\bar x$ serves as initial condition, at $t=0$, for the ODE \refeq{ODE},
not as initial condition for the flow equation \refeq{F}.)

\begin{proof}[Proof of Theorem~\ref{thm:flow}(i)]
  Let $C_S$ be the constant of Lemma~\ref{lem:DPhiyr}, fix $\delta >0$ such that
  $\b>2\delta$ and $1-b > 2\delta$, and
  define $\domV_* = C_S M/((\b-2\delta) \wedge (1-\b-2\delta))$.
  As in the statement of the theorem, assume $\domV > \domV_*$ with this value of $\domV_*$.
  For $t\in [0,1]$ and $x \in \x + \u\B$, let
  \begin{equation} \lbeq{F-def}
    F(t, x) = S(t, x )\rho(x) .
  \end{equation}
  Let $(\Kr_0,\g_0)=(K_0',g_0')$.
  Lemmas~\ref{lem:R-bounds} and \ref{lem:DPhiyr} imply
  that if $\g_0>0$ is sufficiently small then
  $F: [0,1] \times (\x + \u\B) \to X^\w$ is continuously Fr\'echet differentiable and
  \begin{equation} \lbeq{F-bound}
    \|F(t,x)\|_{X^\w}
    \leq \|S(t,x)\|_{L(X^\v,X^\w)} \|\rho(x)\|_{X^\v}
    \leq
    \frac{C_S M}{\hV} \leq (\b-2\delta) \wedge (1-\b-2\delta)
    .
  \end{equation}
  Similarly, by the product rule, there exists $C$ such that
  \begin{multline}
    \|D_xF(t,x)\|_{L(X^\w, X^\w)}
    \leq
    \|[D_xS(t,x)]\rho(x)\|_{L(X^\w, X^\w)}
    \\
    +
    \|S(t,x)[D_x\rho(x)]\|_{L(X^\w, X^\w)}
    \leq
    C
    ,
    \lbeq{DF-bound}
  \end{multline}
  and thus, in particular, $F$ is Lipschitz continuous in $x \in \x+\u\B$.

  We can now apply the standard
  local existence theory for ODE in Banach spaces, as follows.
  For $y \in \u\B$, let
  \begin{equation}
  \lbeq{Fringdef}
    \mathring{F}(t,y) = F(t, \x + y).
  \end{equation}
  Let $X^\w_0= \{ y \in X^\w: \proj_u y_0 = 0 \}$ and $\B_0 = \B \cap X^\w_0$.
  Then  Lemma~\ref{lem:DPhiyr}(i)
  and \eqref{e:F-bound} imply that
  $\mathring{F}(t,(\b-2\delta)\B_0) \subseteq \mathring{F}(t,\u\B_0) \subseteq (\b-2\delta)\B_0$.
  Let $\bar\I$ be the neighbourhood of
  $u_0$ defined by Lemma~\ref{lem:xbarderiv}
  with the same $\delta$ so that $\bar x: \bar\I \to \x + \delta \B$.
  With \eqref{e:F-bound}--\eqref{e:DF-bound},
  the local existence theory for ODEs on Banach spaces
  \cite[Chapter 2, Lemma 1]{AM78}
  implies that the initial value problem
  \begin{equation} \label{e:IVPring}
    \dot y = \mathring{F}(t, y ), \quad y(0) = \bar x(u_0)-\x
  \end{equation}
  has a unique $C^1$-solution $y : [0,1] \to X^\w_0$ 
  such that $y(t) \in (\b-2\delta+\delta)\B_0 = (\b-\delta)\B_0$ for all $t\in[0,1]$.
  In particular,
  \cite[Chapter 2, Lemma 1]{AM78} implies that
  the length of the existence interval of the initial value problem
  \refeq{IVPring} in $(\b-\delta)\B$ is bounded from below by
  $(\b-2\delta)/((\b-2\delta)\wedge (1-\b-2\delta)) \geq 1$ since
  $\|\mathring{F}(t,y)\| \leq (\b-2\delta) \wedge (1-\b-2\delta)$ when $\|y-y(0)\| \leq \b-2\delta$.
  It does not depend on the Lipschitz constant of $\mathring{F}$.

  As discussed around \eqref{e:ODE-F},
  it follows that $x= \x + y(1)$ is a solution to \eqref{e:flow}.
  By construction, $\proj_u x_0=\proj_u\x_0  + \proj_u y(0) = \mathring{u}_0 + (u_0 - \mathring{u_0}) = u_0$.
  Also, $\proj_v y_\infty(1) =0$ because $y(1) \in X^{\w}$, and since
  $\proj_v \x_\infty =\proj_v \bar x_\infty(u_0) =0$, it is also true that
  $\proj_v x_\infty =0$.
  Thus $x$ satisfies the required boundary conditions.

  To prove the estimates \refeq{VVbar1}--\refeq{VVbar4} for $x(u_0)$ with
  $u_0 \in \I\subseteq \bar\I$, we apply
  $\|x(u_0)-\x\|_{X^\w} \leq \b-2\delta$ and $\|\bar x(u_0)-\x\|_{X^\w} \leq \delta$
  to see that
  \begin{equation} \label{e:KKbar}
    \|K_j-\bar K_j\|_{\Wcal_j} \leq \|K_j-\Kr_j\|_{\Wcal_j} + \|\Kr_j-\bar K_j\|_{\Wcal_j}
    \leq (\b-\delta)(a-a_*) \g_j^3,
  \end{equation}
  and analogously that
  \begin{align}
    |g_j-\gbar_j| &\leq (\b -  \delta) \domV\g_j^2 |\log \g_j^2|,\\
    |z_j-\zbar_j| &\leq (\b - \delta) \domV\chi_j\g_j^2 |\log \g_j^2|,\\
    |\mu_j-\mubar_j| &\leq (\b - \delta) \domV\chi_j\g_j^2 |\log \g_j^2|.
  \end{align}
  Since $b - \delta < b$, by assuming that $|\g_0-\gbar_0|$ is sufficiently small,
  i.e., shrinking $\bar \I$ to a smaller neighbourhood $\I$ if necessary,
  we obtain with \refeq{xbar-derivatives} that
  \begin{equation} \lbeq{gringgbar}
    (b - \delta) \g_j^2 |\log \g_j^2| \leq \b \gbar_j^2 |\log \gbar_j^2|.
  \end{equation}
  The required shrinking is uniform on compact subsets of $g_0>0$.
  With the property of the neighbourhood $\bar\I$ stated below \eqref{e:barx-dg0-bd},
  this shows the assertion of Remark~\ref{rk:Nrad}.

  To prove uniqueness, suppose that $x^*$ is a solution to \eqref{e:flow}
  with boundary conditions $(K_0^*,g_0^*)=(K_0,g_0)$ and $(z^*_\infty,\mu^*_\infty)=(0,0)$
  that satisfies \eqref{e:VVbar1}--\eqref{e:VVbar4} (with $x$ replaced by $x^*$,
  and with $\bar x$ as before). Let $\x = \bar x(K_0',g_0')$ as before.
  By assumption and an argument analogous to that given around
  \eqref{e:KKbar}--\eqref{e:gringgbar}, $x^* -\x \in (\b+2\delta)\B_0$.
  It follows that
  $F: [0,1] \times (x^* + (1-\b-2\delta)\B_0) \to X^\w$ is
  Fr\'echet differentiable
  and $\|F(t,x)\|_{X^\w} \leq 1-\b-2\delta$ for all $t \in [0,1]$ and for all
  $x \in x' + (1-\b-2\delta)\B_0 \subset \x + \u\B_0$,
  as discussed around \refeq{F-def}--\refeq{DF-bound}.
  By considering the ODE backwards in time, which is equally well-posed,
  there is a unique solution $x^*(t)$ for $t \in [0,1]$ to $\dot x^* = F(t, x^*)$
  with $x^*(1)=x^*$ and $x^*(t) \subset \x+\u\B_0$.
  It follows that $x^*(0)$ is a flow of $\Phi^0=\bar\Phi$ with the same boundary conditions as
  $\x$. The uniqueness of such flows, by Lemma~\ref{lem:greekbds}, implies that $x^*(0)=\x$, and
  the uniqueness of solutions to the initial value problem \eqref{e:IVPring}
  in $\x + \u\B_0$ then also implies that
  $x = x^*$ as claimed.
  This completes the proof of Theorem~\ref{thm:flow}(i).
\end{proof}

\begin{proof}[Proof of Theorem~\ref{thm:flow}(ii)]
  By Lemma~\ref{lem:xbarderiv}, the map $\bar x: \I \subset \bar\I \to \x + \delta \B \subset X^\w$ is
  continuously Fr\'echet differentiable.
  It therefore follows from \cite[Chapter~2, Lemma~4]{AM78}
  that the solution to the initial value problem \refeq{IVPring}
  is continuously Fr\'echet differentiable in the initial condition.
  To denote the dependence of the solution on the latter,
  we write $y : [0,1] \times \I \to X^\w_0$.
  Let $x(u_0)=\x+y(1,u_0)$, as before.

  By Proposition~\ref{prop:approximate-flow},
  $\bar V_j$ is continuously differentiable in $g_0$ for each $j \in \N_0$.
  Note also that $\bar V_j$ is independent of $K_0$. It
  can be concluded from the differentiability of $\bar V_j$
  and from (A3) that $\bar K_j$ is continuously Fr\'echet differentiable in $(K_0,g_0)$.
  Together with the continuous differentiability of $y$ in the sequence space $X^\w$, this
  implies that as elements of the spaces $X_j$, each $x_j = (K_j,V_j)$ is a $C^1$ function of $u_0$.
  To prove that the derivatives of $z_0$ and $\mu_0$ with respect to $g_0$ are uniformly bounded,
  it suffices to verify this for the contributions to $x$ due to $y$, by Lemma~\ref{lem:barflowderiv}.
  To this end, observe that
  \begin{equation}
    \frac{d}{dt} (Dy)(t) =  D_x\mathring{F}(t, \bar x + y(t)) \circ Dy,  \quad   Dy(0)= {\rm id}.
  \end{equation}
  Thus, by Lemma~\ref{lem:DPhiyr} and Gronwall's inequality
  \cite[Chapter~2, Lemma~2]{AM78},
  \begin{equation}
    \left\|D_{g_0} y(t,K_0,g_0)\right\|_{X^\w}
    \leq C \left\|D_{g_0} \bar x(K_0,g_0)\right\|_{X^\w}
    .
  \end{equation}
  With Lemma~\ref{lem:xbarderiv}, this gives
  \begin{equation} \lbeq{ydu}
    \left\|D_{g_0} y(t,K_0,g_0)\right\|_{X^\w}
    \leq O(\g_0^{-2} |\log \g_0|^{-1}).
  \end{equation}
  Since $\partial\zbar_0/\partial g_0 = O(1)$ and $\partial\mubar_0/\partial g_0=O(1)$
  by Lemma~\ref{lem:barflowderiv}, it follows from \refeq{ydu} and the definition
  of the weights \refeq{weights} that
  \begin{equation}
    \ddp{z_0}{g_0} = O(1), \quad
    \ddp{\mu_0}{g_0} = O(1).
  \end{equation}
  This completes the proof of Theorem~\ref{thm:flow}(ii).
\end{proof}

\section{Proof of Lemma~\ref{lem:DPhiyr}}
\label{sec:pf-DPhiyr}

It now remains only
to prove the key
Lemma~\ref{lem:DPhiyr}.
The proof proceeds in three steps.  The first two steps concern an
approximate version of \refeq{DPhiyr}
and the solution of the
approximate equation, and the third step treats \refeq{DPhiyr} as a
small perturbation of this approximation.

\subsection{Step 1. Approximation of the linear equation}
\label{sec:step1}

Define the map $\bar\Phi_j^0 :X_j \to X_{j+1}$ by extending the map $\bar\varphi_j : \Vcal \to \Vcal$
trivially to the $K$-component, i.e.,
$\bar\Phi_j^0 = (0, \bar\varphi_j)$ in the decomposition $X_{j+1} = \Wcal_{j+1} \oplus \Vcal$.
Thus $\Phi(t, x) = \bar\Phi^0(x) + (\psi(x),t\rho(x))$.
Explicit computation of
the derivative of $\bar\varphi_j$ of \eqref{e:F},
using \eqref{e:Mmatrix}, shows that
\begin{equation}
  \lbeq{DbarPhi}
  D\bar\Phi^0_j(x_j)
  =
  \left( \begin{array}{cc|cc}
      0 & 0 & 0 & 0
      \\
      0 & 1-2\beta_j g_j &   0 & 0
      \\ \hline
      0 & - \xit_j  & 1-\zeta_j g_j & 0
      \\
      0 & \etat_j  &  \gammat_j &  \lambdat_j 
    \end{array}\right)
    ,
\end{equation}
with
\begin{align}
    \etat_j & = \eta_j - 2\upsilon_j^{gg} g_j - \upsilon_j^{gz}z_j - \upsilon_j^{g\mu} \mu_j,
    \nnb
    \gammat_j & = \gamma_j - \upsilon_j^{gz}g_j - 2\upsilon_j^{zz}z_j - \upsilon_j^{z\mu}\mu_j,
    \nnb
    \lambdat_j &= \lambda_j - \upsilon_j^{g\mu}g_j -\upsilon_j^{z\mu}z_j,
    \nnb
    \xit_j &= 2\theta_j g_j + \zeta_j z_j.
    \lbeq{greektilde}
\end{align}
The block matrix structure in \eqref{e:DbarPhi} is in the decomposition $X_j = E_j \oplus F_j$
introduced in Section~\ref{sec:odepf}.
The matrix $D\bar\Phi^0_j(x_j)$ depends on $x_j \in X_j$, but
it is convenient to approximate it by the constant matrix
\begin{equation}
  \lbeq{Ldef}
  L_j
  =
  D\bar\Phi^0_j(\x_j)
  =
  \begin{pmatrix}
    A_j & 0 \\
    B_j & C_j
  \end{pmatrix}
  ,
\end{equation}
where the $2\times 2$ blocks $A_j$, $B_j$, and $C_j$ of $L_j$ are defined by
\begin{gather}
  A_j =
  \begin{pmatrix}
    0 & 0 \\
    0 & 1-2\beta_j \g_j
  \end{pmatrix}
  ,
  \;\;
  B_j =
  \begin{pmatrix}
    0 & - \xir_j \\
    0 & \etar_j
  \end{pmatrix}
  ,
  \;\;
  C_j =
  \begin{pmatrix}
    1-\zeta_j \g_j & 0 \\
    \gammar_j & \lambdar_j
  \end{pmatrix}
  \lbeq{ABC}
\end{gather}
with $\etar_j$, $\gammar_j$, $\lambdar_j$,
and $\xir_j$ as in \eqref{e:greektilde} with $x$ replaced by $\x$.
Thus we study the equation
\begin{equation} \lbeq{Lyf}
  y_{j+1} = L_j y_j + r_j,
\end{equation}
which approximates \eqref{e:DPhiyr}.
To analyse \refeq{Lyf}, and also for later purposes, we derive properties of the
matrices $A_j,B_j,C_j$ in the following lemma.

\begin{lemma}
  \label{lem:Cprodbd}
  Assume (A1--A2). Let $\alpha \in (\lambda^{-1},1)$.
    Then for $\g_0> 0$ sufficiently small (depending on $\alpha$),
    the following hold.
  \begin{enumerate}[(i)]
  \item
  Uniformly in all $l \leq j$,
  \begin{equation} \lbeq{Aprodbd}
    A_j \cdots A_l
    = \begin{pmatrix}
      0 & 0 \\
      0 & O(\g_{j+1}^2/\g_{l}^2)
    \end{pmatrix}.
  \end{equation}
  \item
  Uniformly in all $j$,
  \begin{equation} \lbeq{Bbd}
    B_j = \begin{pmatrix}
      0 & O(\chi_j\g_j) \\
      0 & O(\chi_j)
    \end{pmatrix}
    .
  \end{equation}
  \item
  Uniformly in all $l \geq j$,
  \begin{equation} \lbeq{Cprodbd}
    C_j^{-1} \cdots C_l^{-1}
    = \begin{pmatrix}
      O(1) & 0  \\
      O(\chi_j) & O(\alpha^{l-j+1})
    \end{pmatrix}
    .
  \end{equation}
\end{enumerate}
\end{lemma}

\begin{proof}
(i) It follows immediately from \refeq{ABC} that
\begin{equation} \lbeq{Aprod}
  A_{j} \cdots A_{l} = \prod_{k=l}^{j}(1-2\beta_k \g_k) \proj_g,
\end{equation}
and thus \refeq{prodbd} implies (i).

\smallskip \noindent
(ii) It follows directly from \eqref{e:ABC}
and Lemma~\ref{lem:greekbds} that \eqref{e:Bbd} holds.

\smallskip \noindent
(iii) Note that
\begin{equation}
  \begin{pmatrix}
    c_1 & 0 \\
    b_1 & a_1
  \end{pmatrix}
  \cdots
  \begin{pmatrix}
    c_n & 0 \\
    b_n & a_n
  \end{pmatrix}
  =
  \begin{pmatrix}
    c^* & 0 \\
    b^* &  a^*
  \end{pmatrix}
\end{equation}
with
\begin{equation}
  a^* = a_1 \cdots a_n
  ,
  \quad
  b^* = \sum_{i=1}^n a_1 \cdots a_{i-1} b_{i} c_{i+1} \cdots c_n
  ,
  \quad
  c^* = c_1 \cdots c_n
  .
\end{equation}
We apply this formula with the inverse matrices
\begin{equation}
  C_j^{-1} =
  \begin{pmatrix}
    (1-\zeta_j \g_j)^{-1} & 0 \\
    -(1- \zeta_j \g_j)^{-1} \gammar_j \alr_j & \alr_j
  \end{pmatrix}
\end{equation}
where $\alr_j = \lambdar_j^{-1}$. Thus
\begin{equation}
  C_j^{-1} \cdots C_l^{-1}
  =
  \begin{pmatrix}
    \taur_{j,l} & 0  \\
    \sigmar_{j,l} & \alr_{j,l}
  \end{pmatrix}
\end{equation}
with
\begin{gather}
  \alr_{j,l}
  = \alr_j \cdots \alr_l,
  \qquad
  \taur_{j,l} = \prod_{k=j}^l (1-\zeta_k\g_k)^{-1},
  \\
  \sigmar_{j,l}
  = \sum_{i=1}^{l-j+1} \left( \prod_{k=j+i}^{l} (1-\zeta_k \g_k)^{-1} \right)
  (-\gammar_{j+i-1}) \left( \prod_{k=j}^{j+i-2} \alr_{k} \right)
  .
\end{gather}
The product defining $\taur_{j,l}$ is $O(1)$ by \refeq{zetaprod}.
Assume that $\g_0$ is sufficiently small that, with Lemma~\ref{lem:greekbds} and (A2), $\alr_m < \alpha$ for all $m$. Then
$\alr_{j,l} \le O(\alpha^{l-j+1})$.
Similarly, since $\gammar_m \le O(\chi_m)$,
\begin{equation}
\lbeq{pibarbd}
  |\sigmar_{j,l}|
  \leq  \sum_{i=1}^{l-j+1} \alpha^{i}  O(\chi_{j+i-1})
  \leq O(\chi_j)
  .
\end{equation}
This completes the proof.
\end{proof}

The following lemma provides a solution to \refeq{Lyf}.

\begin{lemma}
  \label{lem:fp}
  Assume (A1--A2) and that $\g_0 > 0$ is sufficiently small.
  We write $y$ as a column vector $y=(u,v)$.
  Then
  \begin{align}
  \lbeq{ucomp}
    u_j & = \sum_{l=0}^{j-1} A_{j-1} \cdots A_{l+1}\proj_ur_l
    \\
  \lbeq{vcomp}
    v_j & =
    - \sum_{l=j}^{\infty} C_{j}^{-1} \cdots C_{l}^{-1} (B_lu_l + \proj_vr_l)
  \end{align}
  is the unique solution to \refeq{Lyf}  which obeys the boundary
  conditions $u_0=v_\infty=0$ and for which the series \refeq{vcomp} converges.
\end{lemma}

The lemma indeed solves \refeq{Lyf}:
given $r$ we first obtain $u$ via \refeq{ucomp}
and then insert $u$ into \refeq{vcomp} to obtain $v$.
The empty product in \refeq{ucomp}
is interpreted as the identity,
so the term in the sum corresponding to $l=j-1$ is
simply $\proj_u r_{j-1}$.

\begin{proof}[Proof of Lemma~\ref{lem:fp}.]
The $u$-component of \refeq{Lyf} is given by
\begin{equation}
  u_{j+1} = A_j u_j + \proj_ur_j.
\end{equation}
By induction, under the initial condition $u_0=0$ this recursion
is equivalent to \refeq{ucomp}.

The $v$-component of \refeq{Lyf} states that
\begin{equation}
  v_{j+1}  = B_ju_j + C_jv_j + \proj_v r_j,
\end{equation}
which is equivalent to
\begin{equation}
  v_j = C_j^{-1} v_{j+1} - C_j^{-1}B_ju_j -C_j^{-1} \proj_v r_j.
\end{equation}
By induction, for any $k \geq j$, the latter is equivalent to
\begin{equation} \lbeq{v-inteq-nolim}
  v_j
  =
  C_j^{-1} \cdots C_{k}^{-1}v_{k+1}
  - \sum_{l=j}^{k} C_{j}^{-1} \cdots C_{l}^{-1} (B_lu_l + \proj_vr_l).
\end{equation}
By Lemma~\ref{lem:Cprodbd}(iii),
with some $\alpha \in (\lambda^{-1}, 1)$
and with $\g_0$ sufficiently small,
$\|C_0^{-1}\cdots C_k^{-1}\|$ is uniformly bounded.
Thus, if $y_j = (u_j,v_j)$ satisfies \eqref{e:Lyf} and $v_j \to 0$,
then $C_0^{-1}\cdots C_k^{-1}v_{k+1} \to 0$
and hence
\begin{equation} \lbeq{v-inteq-lim}
  v_j
  =
  - \sum_{l=j}^{\infty} C_{j}^{-1} \cdots C_{l}^{-1} (B_lu_l + \proj_vr_l),
\end{equation}
which is \refeq{vcomp}.
\end{proof}

\subsection{Step 2. Solution operator for the approximate equation}

We now prove existence, uniqueness, and bounds for the solution to
the approximate equation \eqref{e:Lyf}.

\begin{lemma}
\label{lem:Lyf-bounds}
  Assume (A1--A2) and that $\g_0 > 0$ is sufficiently small.
  For each $r \in X^{\v}$ and $x \in \x + \u\B$, there exists a unique solution
  $y = (u,v) = \Sb r \in X^{\w}$
  to \eqref{e:Lyf} obeying the boundary conditions
  $\proj_u y_0 = 0$, $\proj_v y_\infty = 0$.
  The solution operator $\Sb$ is block diagonal in the decomposition $x=(K,V)$, with
    \begin{equation} \lbeq{Sdiag}
      \Sb = \begin{pmatrix}
        1 & 0 \\
        0 & \Sb_{VV}
      \end{pmatrix}.
    \end{equation}
    There is a constant $C_{\Sb}>0$, such that, uniformly in small $\g_0$,
    \begin{equation} \lbeq{betterbds}
      \|\Sb_{VV}\|_{L(X^{\v}, X^{\w})}
      \leq C_{\Sb}.
    \end{equation}
    The constant $C_{\Sb}$ is independent of the parameters $\domK,\domV$
    which define the domain $D_j$ in \refeq{Djdef}.
\end{lemma}

\begin{proof}
By Lemma~\ref{lem:fp}, it suffices to prove that
the map $r \mapsto y$ defined by \refeq{ucomp}--\refeq{vcomp} gives
a bounded map $\Sb : X^\v \to X^\w$.
For this we use
Lemma~\ref{lem:elementary-recursion}(ii),
from which we recall that for all $k \ge j \ge 0$ and $m \ge 0$,
\begin{equation}
  \lbeq{chigbd}
  \sum_{l=j}^k \chi_l \g_l^n |\log \g_l|^m \leq
  C_{n,m}
  \begin{cases}
    |\log \g_k|^{m+1}  & n = 1
    \\
    \chi_j \g_j^{n-1} |\log \g_j|^m  & n > 1.
  \end{cases}
\end{equation}

\smallskip
\noindent (i) $K$-component.
Since $\proj_K A_l = 0$, we have $\proj_K u_j = \proj_K r_{j-1}$.  Therefore,
by \refeq{weights} and \refeq{gbarmon},
\begin{equation} \lbeq{TK-bound}
  \|\proj_K y\|_{X^\w}
  \leq
  \sup_j \|\proj_K r_{j-1}\|_{X^\w_j}
  \leq
  \sup_j \big[ \w_{K,j}^{-1} \v^{\vphantom{-1}}_{K,j-1} \big] \|r\|_{X^\v}
  \le 2 \|r\|_{X^\v}.
\end{equation}

\smallskip
\noindent (ii) $g$-component.
By Lemma~\ref{lem:Cprodbd}(i), \refeq{weights}, \refeq{gbarmon}, and \refeq{chigbd},
\begin{align} \lbeq{Tg-bound}
  \|\proj_g y\|_{X^\w}
  &\leq
  \sup_j \w_{g,j}^{-1}\sum_{l=0}^{j-1} |\proj_g A_{j-1} \cdots A_{l+1} \pi_u r_l|
  \nnb & \le
  \sup_j \w_{g,j}^{-1}\sum_{l=0}^{j-1}  \v_{g,l} O(\g_{j}/\g_{l})^2  \|r\|_{X^\v}
  \nnb
  &
  \leq
  c \|r\|_{X^\v} \sup_j |\log \g_j|^{-1}   \sum_{l=0}^{j-1} \chi_l \g_l
  \leq
  c \|r\|_{X^\v}
  .
\end{align}

\smallskip
\noindent (iii) $z$-component.
By \refeq{Bbd}--\refeq{Cprodbd}, \refeq{Tg-bound},
\eqref{e:weights}, and \eqref{e:chigbd},
\begin{align}
  \|\proj_z y\|_{X^\w}
  &\leq
  \sup_j \w_{z,j}^{-1} \sum_{l=j}^\infty |\proj_z C_j^{-1} \cdots C_l^{-1}
  (B_lu_l + \pi_v r_l)|
  \nnb
  &\leq
  \sup_j 
  \w_{z,j}^{-1}
  \sum_{l=j}^\infty
  O(1)\left(\chi_l\g_l \w_{g,l}\|r\|_{X^\v} + \chi_l\v_{z,l} \|r\|_{X^\v}
  \right)
  \leq
  c \|r\|_{X^\v}
  .
\lbeq{Tz-bound}
\end{align}

\smallskip
\noindent (iv) $\mu$-component.
We begin with
\begin{align} \lbeq{Tmu-bound}
  \|\proj_\mu y\|_{X^\w}
  &\leq \sup_j
  \w_{\mu,j}^{-1}
  \sum_{l=j}^\infty |\proj_\mu C_j^{-1} \cdots C_l^{-1} (B_lu_l + \pi_v r_l) |.
\end{align}
It is an exercise in matrix algebra, using \refeq{Bbd}--\refeq{Cprodbd}
and \refeq{Tg-bound}, to see
that
\begin{align}
    |\proj_\mu C_j^{-1} \cdots C_l^{-1} B_lu_l |
    & \le
    O(1)\left(  \chi_l \g_l + \alpha^{l-j+1} \right) \w_{g,l} \|r\|_{X^{\v}}
    ,
    \\
    |\proj_\mu C_j^{-1} \cdots C_l^{-1}  \pi_v r_l |
    & \le
     O(1)\left( \chi_j \v_{z,l} + \alpha^{l-j+1} \v_{\mu,l} \right)  \|r\|_{X^{\v}}
     .
\end{align}
Now we use \refeq{weights}, \refeq{chigbd}, and also
\begin{equation}
    \sum_{l=j}^\infty \alpha^{l+1-j}
    \chi_l \g_l^n |\log \g_l|^m  \le O( \chi_j \g_j^n |\log \g_l|^m),
\end{equation}
to conclude that
$\|\proj_\mu y\|_{X^\w} \le c \|r\|_{X^\v}$.
This completes the proof.
\end{proof}

\subsection{Step 3. Solution of the linear equation}

Now we prove Lemma~\ref{lem:DPhiyr}, which involves solving
the equation \refeq{DPhiyr} and estimating its solution operator.
In preparation, we make some definitions and prove two preliminary
lemmas.

We rewrite
\eqref{e:DPhiyr} as
\begin{equation}
\lbeq{Wrewrite}
  y_{j+1} = D_x\Phi_j(t, x_j)y_j + r_j = L_jy_j + W_j(t,x_j)y_j + r_j,
\end{equation}
where
\begin{align}
  W_j(t,x_j)
  &= D_x\Phi_j(t, x_j)-L_j.
\end{align}
It is convenient to define an operator $W(t,x)$ on
sequences via $(W(t,x))_0 = 0$ and $(W(t,x))_{j+1} = W_j(t,x)$.
This operator can be written as a block matrix
with respect to the decomposition $x=(K,V)$ as
\begin{equation}
  W(t,x) = \begin{pmatrix}
    W_{KK} & W_{KV} \\
    W_{VK} & W_{VV}
  \end{pmatrix},
\end{equation}
with $W_{\alpha\beta} = \pi_\alpha W(t,x) \pi_\beta$.

\begin{lemma}
\label{lem:WS0}
Fix $\omega \in (\kappa\Omega,1)$.
The map $W$ obeys $W: [0,1] \times (\x + \u\B) \to L(X^{\w}, X^{\v})$,
$W$ is continuously Fr\'echet differentiable,
and  if $x \in \x + \u\B$ then
\begin{gather}
\lbeq{W0-bound}
  \|W_{KK}\|_{L(X^{\w}, X^{\v})} \leq \omega,\quad
  \|W_{VK}\|_{L(X^{\w}, X^{\v})} \leq  C,
  \\
  \|W_{KV}\|_{L(X^{\w}, X^{\v})} \leq o(1), \quad
  \|W_{VV}\|_{L(X^{\w}, X^{\v})} \leq o(1),
  \quad \text{as $\g_0 \to 0$}
  ,
  \lbeq{W1-bound}
  \\
  \|D_{x}W_j(t, x_j)\|_{L(X_j^{\w},L(X_j^{\w}, X_{j+1}^{\v}))}
  \leq
  C.
  \lbeq{W2-bound}
\end{gather}
\end{lemma}

\begin{proof}
By definition,
\begin{align}
  \lbeq{WDL}
  W_j(t,x_j)
  &=
  [D_x\bar\Phi^0_j(x_j)-D_x\bar\Phi^0_j(\x)]
  + D_x (\psi_j(x_j),t\rho_j(x_j)).
\end{align}The first term on the right-hand side of \refeq{WDL}
only depends on the $V$-components, and
is continuously Fr\'echet differentiable
since, by \refeq{DbarPhi},
$D^2\bar\Phi_j^0$ is a constant matrix for each $j$ with coefficients
bounded by $O(\chi_j)$.
Therefore, for $x \in \x+\u\B$ (with weights chosen maximally),
\begin{align}
  \|[D\bar\Phi^0_j(\x_j) - D\bar\Phi^0_j(x_j)]\proj_V\|_{L(X_j^\w,X_{j+1}^\v)}
  &\leq c  \chi_j  \v_{g,j+1}^{-1}\w_{g,j}^2 \|\x_j - x_j\|_{X^\w_j}
  \nnb
  &
  = O(\hV\u\g_0|\log \g_0|^2).
\end{align}
This contributes to the bounds
\eqref{e:W1-bound}, with $\g_0$ taken small enough.

Lemma~\ref{lem:R-bounds} gives bounds on
the second term on the right-hand side of \refeq{WDL}, as well as its
derivative, and with these
the proof of \refeq{W2-bound} is complete.
\end{proof}

\begin{lemma}
\label{lem:WS02}
For $x \in \x + \u\B$, the map $1-\Sb W(t,x)$ has a bounded inverse
in $L(X^\w,X^\w)$.
\end{lemma}

\begin{proof}
  As in \refeq{Sdiag}, we write $\Sb$ as a block matrix
  with respect to the decomposition $x=(K,V)$ as
  \begin{equation}
    \Sb = \begin{pmatrix}
      1 & 0 \\
      0 & \Sb_{VV}
    \end{pmatrix}
    .
  \end{equation}
  By definition, $1-\Sb W(t,x) = A-B$ with
  \begin{equation}
  \lbeq{ABmatrices}
    A =
    \begin{pmatrix}
      1-W_{KK} & 0 \\
      -\Sb_{VV} W_{VK} & 1-\Sb_{VV} W_{VV}
    \end{pmatrix},
    \quad
    B = \begin{pmatrix}
      0 & W_{KV} \\ 0 & 0
    \end{pmatrix}
    .
  \end{equation}
  For any bounded operators
  $A,B$ on a Banach space, with
  $A^{-1}$ bounded and $\|A^{-1}B\| < 1$,
  the operator $A-B= A(1-A^{-1}B)$ has a bounded inverse since $(1-A^{-1}B)^{-1}$ is
  given by its Neumann series.  Thus it suffices to prove that
  the matrices $A,B$ defined in \refeq{ABmatrices} have these two properties.

  By \refeq{W0-bound}--\refeq{W1-bound} (with $\g_0$ sufficiently small),
  $\|W_{KK}\|_{L(X^\w,X^\w)} < 1$ and $\|\Sb_{VV} W_{VV}\|_{L(X^\w,X^\w)} < 1$.
  Thus $A$ is a block matrix of the form
  \begin{equation}
    A = \begin{pmatrix}
      A_{KK} & 0 \\
      A_{VK} & A_{VV} \\
    \end{pmatrix},
  \end{equation}
  where $A_{KK}$ and $A_{VV}$ have inverses in $L(X^\w,X^\w)$.  It follows that $A$
  has the bounded inverse on $X^\w$ given by the block matrix
  \begin{equation}
    A^{-1}
    = \begin{pmatrix}
      A_{KK}^{-1} & 0
      \\
      A_{VV}^{-1} A_{VK} A_{KK}^{-1} & A_{VV}^{-1}
    \end{pmatrix}.
  \end{equation}
  By \refeq{W0-bound}--\refeq{W1-bound} (with $\g_0$ sufficiently small),
  $\|A^{-1}B\|_{L(X^\w,X^\w)} < 1$, and the proof is complete.
\end{proof}

\begin{proof}[Proof of Lemma~\ref{lem:DPhiyr}]
(i)
By the assumption that $y \in X^{\w}$, Lemma~\ref{lem:Lyf-bounds},
and \refeq{W1-bound}, the equation \refeq{Wrewrite}
with the boundary conditions of Lemma~\ref{lem:DPhiyr}(i) is equivalent to
\begin{equation}
\lbeq{ySb}
  y = \Sb W(t,x)y+ \Sb r
  .
\end{equation}
It follows that the solution operator is given by
  \begin{equation} \lbeq{SNeumann}
    S(t,x) = (1-\Sb W(t,x))^{-1}\Sb,
  \end{equation}
with the existence of the inverse operator guaranteed by Lemma~\ref{lem:WS0}.

\smallskip \noindent (ii)
This follows from \refeq{SNeumann}
and Lemmas~\ref{lem:Lyf-bounds} and \ref{lem:WS02}.

\smallskip \noindent (iii)
By \refeq{SNeumann},
continuous Fr\'echet differentiability in $x$ of
$S(t,x)$ follows from the continuous Fr\'echet differentiability of $\Sb W(t,x)$,
which itself follows from part~(i) and from
$D_x\Sb W(t,x) = \Sb D_xW(t,x)$
by linearity of $\Sb$.
Explicitly,
\begin{equation}
  D_x S(t,x)
  =
  (1-\Sb W(t,x))^{-1}D_x\Sb W(t,x)(1-\Sb W(t,x))^{-1}\Sb
  .
\end{equation}
By \refeq{W2-bound},
\begin{equation}
\lbeq{DSW}
  \|D_x\Sb W(t,x)\|_{L(X^\w,L(X^\w, X^\w))}
  \leq C \|D_xW(t,x)]\|_{L(X^\w,L(X^\w, X^\v))}
  \leq C.
\end{equation}
Together with the boundedness of the operators $(1-\Sb W(t,x))^{-1}$ and $\Sb$,
this proves \eqref{e:DStx-bound} and completes the proof.
\end{proof}


\section*{Acknowledgement}
This work was supported in part by NSERC of Canada.
The authors thank the referee for helpful comments.
RB gratefully acknowledges the hospitality
of the Department of Mathematics and Statistics at McGill University,
where part of this work was done.
DB gratefully acknowledges the support and hospitality of
the Institute for Advanced Study at Princeton and of Eurandom during part
of this work.
GS gratefully acknowledges the support and hospitality of
the Institut Henri Poincar\'e, where part of this work was done.

\bibliographystyle{plain}

\end{document}